\documentclass[11 pt]{amsart}
\title{Multipliers and invariants of endomorphisms of projective space in dimension greater than 1}

\author{Benjamin Hutz}
\address{
Department of Mathematics and Statistics\\
Saint Louis University\\
St. Louis, MO}
\email{benjamin.hutz@slu.edu}



\usepackage{color}
\usepackage{amssymb,amsmath,amsthm,fullpage}
\usepackage[all]{xy}  
\usepackage{url}
\usepackage{rotating} 
\usepackage{tikz}
\usepackage{textcomp,listings}
\usepackage{enumerate}
\usepackage[stable]{footmisc} 
\usepackage{versions}

\definecolor{green}{rgb}{0,0.5,0}
\definecolor{dkgreen}{rgb}{0,0.6,0}
\definecolor{gray}{rgb}{0.5,0.5,0.5}
\definecolor{mauve}{rgb}{0.58,0,0.82}
\lstnewenvironment{python}[1][]{
\lstset{
  frame=single,                   
  language=python,                          
  basicstyle=\ttfamily\small, 
  numbers=left,                             
  numberstyle=\scriptsize\color{black},  
  stepnumber=1,                   
  numbersep=7pt,                  
  backgroundcolor=\color{white},      
  stringstyle=\color{red},
  showspaces=false,               
  showstringspaces=false,         
  showtabs=false,                 
  tabsize=2,                      
  captionpos=b,                   
  breaklines=true,                
  breakatwhitespace=false,        
  rulecolor=\color{black},        
  framexleftmargin=1mm, framextopmargin=1mm, frame=shadowbox, 
  alsoletter={1234567890},
  otherkeywords={\ , \}, \{},
  keywordstyle=\color{blue},       
  stringstyle=\color{mauve},         
  emph={access,and,break,class,continue,def,del,elif ,else,%
  except,exec,finally,for,from,global,if,import,in,i s,%
  lambda,not,or,pass,print,raise,return,try,while},
  emphstyle=\color{black}\bfseries,
  emph={[2]True, False, None, self},
  emphstyle=[2]\color{blue},
  emph={[3]from, import, as},
  emphstyle=[3]\color{blue},
  upquote=true,
  morecomment=[s]{"""}{"""},
  commentstyle=\color{dkgreen},       
  emph={[4]1, 2, 3, 4, 5, 6, 7, 8, 9, 0},
  emphstyle=[4]\color{blue},
  literate=*{:}{{\textcolor{blue}:}}{1}%
  {=}{{\textcolor{blue}=}}{1}%
  {-}{{\textcolor{blue}-}}{1}%
  {+}{{\textcolor{blue}+}}{1}%
  {*}{{\textcolor{blue}*}}{1}%
  {!}{{\textcolor{blue}!}}{1}%
  {(}{{\textcolor{blue}(}}{1}%
  {)}{{\textcolor{blue})}}{1}%
  {[}{{\textcolor{blue}[}}{1}%
  {]}{{\textcolor{blue}]}}{1}%
  {<}{{\textcolor{blue}<}}{1}%
  {>}{{\textcolor{blue}>}}{1},%
}}{}

\definecolor{orange}{rgb}{1,0.65,0.17}

\def\C{\mathbb{C}}

\def\Z{\mathbb{Z}}
\def\Q{\mathbb{Q}}

\def\P{\mathbb{P}}
\def\A{\mathbb{A}}

\def\O{\mathcal{O}}

\def\calM{\mathcal{M}}

\DeclareMathOperator{\Gal}{Gal} \DeclareMathOperator{\Hom}{Hom}
\DeclareMathOperator{\Per}{Per} 
\DeclareMathOperator{\PGL}{PGL} \DeclareMathOperator{\GL}{GL}

\DeclareMathOperator{\Fix}{Fix}

\DeclareMathOperator{\SL}{SL}

\DeclareMathOperator{\Res}{Res}

 \DeclareMathOperator{\Proj}{Proj}
 
\DeclareMathOperator{\charpoly}{charpoly}
\DeclareMathOperator{\Part}{Part}

\theoremstyle{plain}
\newtheorem{thm}{Theorem}[section]
\newtheorem*{thm*}{Theorem}
\newtheorem*{thmA}{Theorem A}
\newtheorem*{thmB}{Theorem B}
\newtheorem*{thmC}{Theorem C}
\newtheorem{lem}[thm]{Lemma}
\newtheorem{prop}[thm]{Proposition}
\newtheorem*{prop*}{Proposition}
\newtheorem{cor}[thm]{Corollary}
\theoremstyle{definition}
\newtheorem{defn}[thm]{Definition}
\newtheorem{conj}[thm]{Conjecture}
\newtheorem{exmp}[thm]{Example}
\newtheorem*{exmp*}{Example}

\newtheorem{algo}[thm]{Algorithm}
\theoremstyle{remark}
\newtheorem*{rem}{Remark}

\excludeversion{code}

\begin{document}
\maketitle

\begin{abstract}
     There is a natural conjugation action on the set of endomorphism of $\P^N$ of fixed degree $d \geq 2$. The quotient by this action forms the moduli of degree $d$ endomorphisms of $\P^N$, denoted $\calM_d^N$. We construct invariant functions on this moduli space coming from to set of multiplier matrices of the periodic points. The basic properties of these functions are demonstrated such as that they are in the ring of regular functions of $\calM_d^N$, methods of computing them, as well as the existence of relations. The main part of the article examines to what extend these invariant functions determine the conjugacy class in the moduli space. Several different types of isospectral families are constructed and a generalization of McMullen's theorem on the multiplier mapping of dimension 1 is proposed. Finally, this generalization is shown to hold when restricted to several specific families in $\calM_d^N$.
\end{abstract}

\section{Introduction}
    Let $f:\P^N \to \P^N$ be a degree $d \geq 2$ endomorphism, i.e., defined by an ($N+1$)-tuple of homogeneous degree $d$ polynomials with no common zeroes. We define the $n$th iterate of $f$ for $n \geq 1$ as $f^n = f \circ f^{n-1}$ with $f^0$ the identity map. Let $\Hom_d^N$ be the set of all such endomorphisms. There is a natural conjugation action of $\Hom_d^N$ by elements of $\PGL_{N+1}$ (the automorphism group of $\P^N$), and we denote $f^{\alpha} = \alpha^{-1} \circ f \circ \alpha$ for $\alpha \in \PGL_{N+1}$. Conjugation preserves the intrinsic properties of a dynamical system, so we consider the quotient $\calM_d^N = \Hom_d^N/\PGL_{N+1}$. This quotient exists as a geometric quotient ($N=1$: \cite{Milnor,Silverman9}, $N > 1$: \cite{Levy, petsche}) and is called the \emph{moduli space of dynamical systems of degree $d$ on $\P^N$}. These moduli spaces (and their generalizations) have received much study; for example, see \cite{DeMarco, JDoyle2, Manes3, McMullen2, Milnor, Silverman9}.

    Note that the action of $\PGL_{N+1}$ on $\Hom_d^N$ induces an action on the ring of regular functions $\Q[\Hom_d^N]$ of $\Hom_d^N$. Hence, it makes sense to talk about functions in $\Q[\Hom_d^N]$ that are invariant under the action of $\PGL_{N+1}$. We denote the set of such functions as $\Q[\Hom_d^N]^{\PGL_{N+1}}$. The action of $\PGL_{N+1}$ can be lifted to an equivalent action of $\SL_{N+1}$ and it is often easier to work with this action. In dimension 1, Silverman \cite{Silverman9} proves that the ring of regular functions of the moduli space is exactly this set of invariant functions, i.e., $\Q[\calM_d^1] = \Q[\Hom_d^1]^{\SL_{2}}$. In Proposition \ref{prop_regular}, we prove the general statement for $\Hom_d^N$ and in Theorem \ref{thm_sigma_regular} prove that the multiplier invariants $\sigma_{i,j}^{(n)}$ defined in Section \ref{sect_mult_inv} are in $\Q[\calM_d^N]$. These results are summarized in the following theorem, with more detailed statements in the main body.
    \begin{thmA}
        The ring of regular functions of the moduli space of degree d dynamical systems on $\P^N$ satisfies $\Q[\calM_d^N] = \Q[\Hom_d^N]^{\SL_{N+1}}$. Furthermore, the $\sigma_{i,j}^{(n)}$ and the $\sigma^{\ast(n)}_{i,j}$ (defined in Section \ref{sect_mult_inv}) are regular functions on $\calM_d^N$ for all appropriate choices of $i,j,n$
    \end{thmA}

    One of the early motivations for studying invariant functions on the moduli spaces comes from Milnor \cite{Milnor}. Milnor proved that $\calM_2^1(\C) \cong \A^2(\C)$ by giving an explicit isomorphism utilizing invariant functions constructed from the multipliers of the fixed points (extended to schemes over $\Z$ by Silverman \cite{Silverman9}). Recall that for a fixed point $z$ of a rational function $\phi$, the multiplier at $z$ is defined as $\lambda_z = \phi'(z)$. Given the multipliers of the three fixed points $\lambda_1$, $\lambda_2$, and $\lambda_3$ of a degree 2 endomorphism of $\P^1$, we can define three invariant functions $\sigma_1$, $\sigma_1$, and $\sigma_3$ as
    \begin{equation*}
        (t-\lambda_1)(t-\lambda_2)(t-\lambda_3) = \sum_{i=0}^3 (-1)^i\sigma_{i}t^{3-i},
    \end{equation*}
    where $t$ is an indeterminate.
    In particular, $\sigma_i$ is the $i$-th elementary symmetric polynomial evaluated on the multipliers of the fixed points. Because the set of multipliers of the fixed points (the multiplier spectrum) are, at worst, permuted under conjugation, the values  $\sigma_1$, $\sigma_2$, and $\sigma_3$ are invariants of the conjugacy class. Milnor's isomorphism is then given explicitly as
    \begin{align*}
        \calM_2^1 &\to \A^2\\
        [f] &\mapsto (\sigma_1,\sigma_2).
    \end{align*}

    Milnor's map can be extended to any degree as
    \begin{align*}
        \tau_{d,1}: \calM_d^1 &\to \A^{d}\\
        [f] &\mapsto (\sigma_1,\ldots,\sigma_{d}).
    \end{align*}
    Note that we utilize only the first $d$ of the $d+1$ fixed point multiplier invariants since there is the following relation (see Hutz-Tepper \cite{Hutz10} or Fujimura-Nishizawa \cite[Theorem 1]{Fujimura})
    \begin{equation*}
        (-1)^{d+1}\sigma_{d+1} + (-1)^{d-1}\sigma_{d-1} + (-1)^{d-2} 2\sigma_{d-2} + \cdots  - (d-1)\sigma_{1} +d=0.
    \end{equation*}
    For degree larger than $2$, the map $\tau_{d,1}$ is no longer an isomorphism. When restricted to polynomials, Fujimura \cite{Fujimura} provided the cardinality of a generic fiber is $(d-2)!$ and Sugiyama gave an algorithm to compute the cardinality of any fiber \cite{Sugiyama2}. However, less is known for rational functions. McMullen \cite{McMullen2} showed that when $\tau_{d,1}$ is extended to include symmetric functions of the multipliers of periodic points of higher period, denoted $\tau_{d,n}$, the resulting map is generically finite-to-one away from the Latt\`es maps. One open question in this area is to determine the cardinality of a generic fiber of this generalization. The author and Michael Tepper studied this problem for degree 3 rational functions \cite{Hutz10} and proved that the cardinality of $\tau_{d,2}$ is generically at most $12$, i.e., when including the multiplier invariants associated to the fixed points and periodic points of period $2$. Levy studied $\tau_{d,n}$ in dimension 1 in positive characteristic \cite{Levy4}.

    In dimension greater than $1$, these problems have received little study. There is a series of papers by Guillot from the complex perspective studying the eigenvalues of the multiplier matrices of the fixed points in certain special cases \cite{Guillot,Guillot3,Guillot4,Guillot2}. The parts of his work most related to this article are mainly in the recent preprint \cite{Guillot4} and concern relations among the eigenvalues of the fixed point multipliers for quadratic self-maps of $\P^2$ as well as how well these eigenvalues are able to determine the map up to linear equivalence. While there is some small overlap of results on $\P^2$ with his sequence of papers (mainly Corollary \ref{cor_monic_finite}), the methods are entirely different and the focus here is on moduli invariants in general.

    With the invariant functions defined in Section \ref{sect_mult_inv}, we are able to conjecture in Section \ref{sect_mcmullen} an analog of McMullen's Theorem for $\calM_d^N$ and $N \geq 2$ (Conjecture \ref{conj_mcmullen}). We then go on to prove in that section that the multiplier map $\tau_{d,n}^N$ is finite-to-one for certain specific families as well as provide a number of examples of infinite families whose image is a single point (similar to the Latt\`es families in dimension $1$). These families, whose multiplier spectrums are the same for every member of the family, are called \emph{isospectral} (see Definition \ref{defn_isospectral}). These results provide some partial answers to questions raised during the Bellairs Workshop on Moduli Spaces and the Arithmetic of Dynamical Systems in 2010 and subsequent notes published by Silverman \cite[Question 2.43]{Silverman20} and also raised in Doyle-Silverman \cite[Question 19.5]{JDoyle2}.

    The following theorem summarizes the results of this article on isospectral families.
    \begin{thmB}
        The following constructions produce isospectral families.
        \begin{enumerate}
            \item (Theorem \ref{iso_symmetric}) Let $f_a:\P^1 \to \P^1$ be a family of Latt\`es maps of degree $d$. Then the $k$-symmetric product $F$ is an isospectral family in $\Hom_d^k$.
            \item (Theorem \ref{thm_cart_prod})
                Let $f_a:\P^N \to \P^N$ and $g_b:\P^M \to \P^M$ be isospectral families of morphisms with $\deg(f_a) = \deg(g_b)$, then the cartesian product family $h_{a,b} = f_a \times g_b$ is isospectral in $\Hom_d^{N+M+1}$.

            \item (Theorem \ref{iso_segre}) Let $f_a:\P^N \to \P^N$ be an isospectral family of degree $d$ and $g:\P^M \to \P^M$ the degree $d$ powering map. Then the family of endomorphisms of $h_{a}:\P^{(N+1)(M+1)-1} \to \P^{(N+1)(M+1)-1}$ induced by the Segre embedding of $f_a \times g$ is isospectral in $\Hom_d^{(N+1)(M+1)-1}$.
        \end{enumerate}
    \end{thmB}
    The following theorem summarizes the results of this article that prove special cases where the multiplier map is finite-to-one (Conjecture \ref{conj_mcmullen}).
    \begin{thmC}
        The multiplier map is finite-to-one when restricted to the following families.
        \begin{enumerate}
            \item (Theorem \ref{finite_split}) The fixed point multiplier map, $\tau_{d,1}^{N}$, is (generically) $((d-2)!)^{N}$-to-one when restricted to split polynomial endomorphisms.
            \item (Theorem \ref{thm_triangular})
                The fixed point multiplier map $\tau_{d,1}^N$ is (generically) finite-to-one when restricted to triangular polynomial endomorphisms.
            \item (Corollary \ref{cor_monic_finite}) The fixed point multiplier map $\tau_{2,1}^2$ is (generically) finite-to-one when restricted to monic polynomials of the form \eqref{eq_normal_form}. The explicit hypersurface given by the image of $\tau_{2,1}^2$ restricted to this family is given in Theorem \ref{thm_monic_poly}.
        \end{enumerate}
    \end{thmC}

    It should be noted that while this article was being prepared John Doyle and Joseph Silverman released an article \cite[Section 19]{JDoyle2} that touches briefly on constructing invariants on $\calM_d^N$. They construct a similar version of invariant functions from multipliers without the associated study of the ring of regular functions $\Q[\calM_d^N]$ and mainly raise a number of interesting questions related to these invariant functions. Consequently, their article can almost be considered as additional motivation for this work as, while there is some overlap of subject, there is little overlap of results. Further, we find our fuller construction of the invariant functions obtained from multipliers more useful in computations such as those in Theorem \ref{thm_monic_poly} due to their lower complexity. The hope is that this initial study of multiplier invariants in higher dimensions will lead to many fruitful results.

    The organization of the article is as follows. Section \ref{sect_mult_inv} defines the multiplier invariants and studies their basic properties. Section \ref{sect_gen} discusses generators and relations among the invariants. Section \ref{sect_computing} gives an algorithm to compute the invariants without computing either the periodic points or their multiplier matrices. Section \ref{sect_mcmullen} conjectures a higher dimensional statement of McMullen's theorem and gives both isospectral families and families that are finite-to-one under $\tau_{d,n}^N$.

\section{Defining the Multiplier Invariants} \label{sect_mult_inv}

    Let $f \in \Hom_d^N$. We designate the \emph{$n$-th iterate} of $f$ with exponential notation $f^n = f \circ f^{n-1}$. For an element $\alpha \in \PGL_{N+1}$, we denote the \emph{conjugate} as $f^{\alpha} = \alpha^{-1} \circ f \circ \alpha$. At each fixed point $P$, the map induced on the tangent space by $f$, $df_P$, is an element of $\GL_N$, after choosing a basis. While the resulting matrix is not independent of this choice of basis, the characteristic polynomial of this matrix is independent of this choice of basis. We will denote the set of fixed points of $f$ as $\Fix(f)$ and the points of period $n$ for $f$ as $\Per_n(f)$.
    \begin{defn}
        Define the \emph{multiplier polynomial} at $P \in \Fix(f)$ to be the characteristic polynomial of $df_P$ denoted $\gamma_{f,P}$. Its eigenvalues will be denoted $\lambda_{P,1}, \ldots, \lambda_{P,N}$. This definition can be extended to periodic points of any order by considering the fixed points of the iterate $f^n$. When needed, we will call $df_P$ the \emph{multiplier matrix at $P$}.

        The \emph{polynomial $n$-multiplier spectrum} of $f$ is the set of multiplier polynomials for the periodic points of period $n$
        \begin{equation*}
            \Gamma_n(f) = \{ \gamma_{f^n,P} : P \in \Per_n(f)\},
        \end{equation*}
        and the \emph{$n$-multiplier spectrum} is the set of eigenvalues (with multiplicity) of the multiplier matrices or, equivalently, the zeros of the multiplier polynomials
        \begin{equation*}
            \Lambda_n(f) = \{\lambda_{P,1},\ldots, \lambda_{P,N} : P \in \Per_n(f)\}.
        \end{equation*}
        The map $f$ will be dropped from the notation, when it is clear what map is being referred to.
    \end{defn}

    When counted with multiplicity, there are $D_n=\frac{d^{n(N+1)}-1}{d-1}$ points of period $n$ for $f$. We can make similar definitions when working with the points of formal period $n$; the designation will be marked by a superscript asterisk; e.g.,
    \begin{equation*}
        \Gamma_n^{\ast}(f) = \{ \gamma_{f,P} : P \in \Per^{\ast}_n(f)\}.
    \end{equation*}
    Since we will only mention the formal periodic point case in passing, we leave the definitions and properties to the references; see Hutz \cite{Hutz1}.

    We first see that multiplier polynomials are invariant under conjugation by an element of $\PGL_{N+1}$.
    \begin{lem}\label{lem_conj_mult}
        Let $f \in \Hom_d^N$ and $P \in \Fix(f)$ and $\alpha \in \PGL_{N+1}$. Then $\gamma_{f,P} = \gamma_{f^{\alpha},\alpha^{-1}(P)}$.
    \end{lem}
    \begin{proof}
        Without loss of generality, we may assume that $P$ and $\alpha^{-1}P$ are in the same affine chart. Let $\boldsymbol{x} = (x_0,\ldots,x_N)$ be coordinates for $\P^N$ and fix $j$ such that the $j$th coordinate of $P$ is not zero. Let $\boldsymbol{z}$ be coordinates of the affine chart obtained by dehomogenizing at $x_j$. Let $\phi$ be the dehomogenization of $f$, and $\tilde{P}$ the dehomogenization of $P$.
        Denoting the Jacobian matrix as $\frac{\partial \phi}{\partial \boldsymbol{z}}$, the chain rule tells us that
        \begin{align*}
            \frac{\partial \phi^{\alpha}}{\partial \boldsymbol{z}}(\alpha^{-1}\tilde{P})
            &= \frac{\partial }{\partial \boldsymbol{z}}(\alpha^{-1} \circ \phi \circ \alpha)(\alpha^{-1}\tilde{P})\\
            &=\frac{\partial \alpha^{-1}}{\partial \boldsymbol{z}}(\phi(\tilde{P})) \cdot \frac{\partial \phi}{\partial \boldsymbol{z}}(\tilde{P}) \cdot \frac{\partial \alpha}{\partial \boldsymbol{z}}(\alpha^{-1}\tilde{P})\\
            &=\frac{\partial \alpha^{-1}}{\partial \boldsymbol{z}}(\tilde{P}) \cdot \frac{\partial \phi}{\partial \boldsymbol{z}}(\tilde{P}) \cdot \frac{\partial \alpha}{\partial \boldsymbol{z}}(\alpha^{-1}\tilde{P}).
        \end{align*}
        For matrices defined over a field, we have $ABC$ and $BCA$ are similar matrices since $BCA = A^{-1} ABC A$. Consequently,
        \begin{equation*}
            \charpoly(ABC) = \charpoly(BCA)
        \end{equation*}
        so that
        \begin{equation*}
            \gamma_{f^{\alpha},\alpha^{-1}P}= \charpoly\left(\frac{\partial \phi}{\partial \boldsymbol{z}}(\tilde{P}) \cdot \frac{\partial \alpha}{\partial \boldsymbol{z}}(\alpha^{-1}\tilde{P})\cdot \frac{\partial \alpha^{-1}}{\partial \boldsymbol{z}}(\tilde{P}) \right).
        \end{equation*}
        Moreover,
        \begin{equation*}
            \frac{\partial \alpha}{\partial \boldsymbol{z}}(\alpha^{-1}\tilde{P}) \cdot \frac{\partial \alpha^{-1}}{\partial \boldsymbol{z}}(\tilde{P})  = \frac{\partial }{\partial \boldsymbol{z}}\left( \alpha \circ \alpha^{-1}\right)(\tilde{P}) = Id.
        \end{equation*}
        Therefore,
        \begin{equation*}
            \gamma_{f^{\alpha},\alpha^{-1}P} = \gamma_{f,P}
            = \charpoly\left(\frac{\partial \phi}{\partial \boldsymbol{z}}(\tilde{P})\right).
        \end{equation*}
    \end{proof}

    The fixed points depend algebraically but not rationally on the coefficients of $f$ and form an unordered set. Furthermore, Lemma \ref{lem_conj_mult} tells us that the multiplier spectra depend only on the conjugacy class of $f$. Consequently, we use the multipliers to define invariants on the moduli space.

    We are looking for functions on the moduli space, we need them to be invariant under conjugation. Conjugation can permute the fixed points. We also need to be careful with the basis for the tangent space. The eigenvalues of the Jacobian matrix are not fixed, but the symmetric functions of the eigenvalues are. So we need to find functions that are invariant under the action of the wreath product $S_N \wr S_{D_n}$, where $S_N$ is the symmetric group on $N$ elements. We can think of the wreath product as acting on $D_n$ sets of $N$ variables (the $N$ eigenvalues of each of the $D_n$ periodic point multiplier matrices). We can permute each set of variables independently and we can permute the sets of variables. In dimension 1, the multiplier matrix is a single complex number and we have $S_1 \wr S_{D_n} \cong S_{D_n}$. Furthermore, the invariant functions arising from the fixed point multipliers are generated by the elementary symmetric polynomials of the fixed point multipliers. In higher dimensions, the invariant ring of the finite group $S_N \wr S_{D_n}$ is much more complicated and can involve both primary and secondary invariants. While there is a reasonably nice description of the primary invariants of wreath products in terms of the primary invariants of the component groups \cite[Theorem 7.10]{Dias}, this seems to lead to an overly complicated system of generators (recall that the number of primary/secondary invariants is not fixed and some choices produce ``better'' sets of generators). Regardless, because of the rapid growth of the size of these wreath products, computing explicit sets of generators is not feasible even for small $N$, $n$, and $d$. We take the following approach.

    Let $t$ be the indeterminant for the characteristic polynomials and consider the polynomial in variables $(w,t)$ given by
    \begin{equation*}
        \Sigma_n(f) = \prod_{P \in \Per_n(f)} w - \gamma_{f,P} = \prod_{i=1}^{D_n} w- \prod_{j=1}^N (t-\lambda_{i,j}).
    \end{equation*}
    Label the coefficients of this polynomial by the complement of the bi-degree in $(w,t)$, i.e.,
    \begin{equation}\label{eq_sigma}
        \Sigma_n(f) = \sum_{i=0}^{D_n}\sum_{j=0}^{Ni} (-1)^{i+j}\sigma^{(n)}_{i,j}w^{D_n-i}t^{Ni-j}.
    \end{equation}
    We can make a similar definition for the formal periodic points:
    \begin{equation} \label{eq_sigma_ast}
        \Sigma_n^{\ast}(f) = \sum_{i=0}^{D^{\ast}_n}\sum_{j=0}^{Ni} (-1)^{i+j}\sigma^{\ast(n)}_{i,j}w^{D_n-i}t^{Ni-j}.
    \end{equation}

    The $\sigma^{(n)}_{i,j}$ are symmetric functions of the (coefficients of the) $\gamma_{P}$, and the (coefficients of the) $\gamma_P$ are rational functions of $P$ and the coefficients of $f$. Consequently, the $\sigma^{(n)}_{i,j}$ are rational functions on $\Hom_d^N$. We now show they are regular functions on the moduli space $\calM_d^N = \Hom_d^N/\PGL_{N+1}$.
    Similar to Silverman \cite{Silverman9}, we generalize this construction to work over $\Z$.

    Let $M = \binom{N+d}{N}$ be the number of monomials of degree $d$ in $N+1$ variables. We can identify $\P^{(N+1)M-1}$ with $N+1$ tuples of homogeneous polynomials of degree $d$. For indeterminants $a_{i,j}$, denote
    \begin{equation} \label{eq_fgen}
        f_{gen} = [f_0,\ldots,f_N] = [a_{00}x_0^d + \cdots + a_{0M}x_N^d, a_{10}x_0^d + \cdots + a_{1M}x_N^d,\ldots,a_{N0}x_0^d + \cdots + a_{NM}x_N^d].
    \end{equation}
    Let $\rho = \Res(f_0,\ldots,f_N) \in \Z[a_{ij}]$ be the Macaulay resultant. The set $\Hom_d^N$ is the open subset of $\P^{(N+1)M-1}$ defined by the condition $\rho \neq 0$. Then
    \begin{equation*}
        \Hom_d^N = \Proj \Z[a_{ij}]/\{\rho = 0\}
    \end{equation*}
    and so
    \begin{equation*}
        H^1(\Hom_d^N, \O_{\Hom_d^N}) = \Z[a_{ij}][\rho^{-1}]_{(0)},
    \end{equation*}
    where the subscript $(0)$ denotes the elements of degree $0$ (i.e., rational functions whose numerator and denominator are homogeneous of the same degree).

    \begin{prop}\label{prop_regular}
        The ring of regular functions $\Q[\Hom_d^N]$ of the affine variety $\Hom_d^N$ is given explicitly by
        \begin{equation*}
            \Q[\Hom_d^N] = \Q\left[\frac{a_{ij}^{e_{ij}}}{\rho} : \sum e_{ij} = (N+1)d^N \right].
        \end{equation*}
        Furthermore, $\Q[\calM_d^N] = \Q[\Hom_d^N]^{\SL_{N+1}}$.
    \end{prop}
    \begin{proof}
        The proof follows directly from generic properties of projective and affine varieties; see, for example, the proof of Proposition 4.27 in Silverman \cite[proposition 4.27]{Silverman10}.
        See Levy \cite{Levy} for the ``furthermore'' statement.
    \end{proof}

    We now show that the $\sigma_{i,j}^{(n)}$ and $\sigma_{i,j}^{\ast(n)}$ are regular functions on $\calM_d^N$. This generalizes the dimension 1 result of Silverman \cite[Theorem 4.50]{Silverman10}.
    \begin{thm}\label{thm_sigma_regular}
        For $f \in \Hom_d^N$ and $n \geq 1$ and $i,j$ in the appropriate range, let $\sigma_{i,j}^{(n)}$ and $\sigma_{i,j}^{\ast(n)}$ be the multiplier invariants associated to $f$ as defined by equations \eqref{eq_sigma} and \eqref{eq_sigma_ast}, respectively.
        \begin{enumerate}
            \item The functions
            \begin{equation*}
                f \mapsto \sigma_{i,j}^{(n)}
                \quad \text{and} \quad
                f \mapsto \sigma_{i,j}^{\ast(n)}
            \end{equation*}
            are in $\Q[a_{kl}, \rho^{-1}]$, where $\{a_{kl}\}$ are the coefficients of a generic element $f_{gen}$ of $\Hom_d^N$ and $\rho$ the resultant of $f_{gen}$.
            \item The functions are $\PGL_{N+1}$ invariant and, hence, are in the ring of regular functions $\Q[\calM_d^N]$.
        \end{enumerate}
    \end{thm}
    \begin{proof}
    \hfill
    \begin{enumerate}
    \item
        From Minimair \cite[Theorem 1]{Minimair} we know that the resultant of an iterate of $f$ is a power of the resultant of $f$. Denote $\rho^{(n)}$ as the resultant of the $n$-th iterate. Then we have
        \begin{equation*}
            \Q[a_{kl}, \rho^{-1}] = \Q[a_{kl}, (\rho^{(n)})^{-1}].
        \end{equation*}
        We can replace $f$ by $f^n$ and consider only the multiplier spectrum of the fixed points.

        Let $K$ denote the field $\Q(a_{kl})$ treating $a_{kl}$ as indeterminants. The fixed points of $f_{gen}$ (defined in equation \eqref{eq_fgen}) are the common zeros of a finite collection of polynomials with coefficients in $K$. Hence, the set of fixed points $\Fix(f)$ and their multipliers $\Lambda_1(f)$ are $\Gal(\overline{K}/K)$ invariant sets. Thus, the symmetric functions $\sigma^{(1)}_{i,j}(f)$ are $\Gal(\overline{K}/K)$ fixed elements of $\overline{K}$, so are in $K$. Furthermore, $\sigma^{(1)}_{i,j}$ is homogeneous in $\{a_{kl}\}$ in the sense that $\sigma^{(1)}_{i,j}$ gives the same values for $f_{gen} = [f_0,\ldots,f_N]$ and $cf_{gen}=[cf_0,\ldots,cf_N]$ for any nonzero constant $c$. In particular, $\sigma^{(1)}_{i,j}$ is in $K^{(0)}$, the set of rational functions of $\{a_{kl}\}$ whose numerator and denominator have the same degree.

        We want to show that these are regular functions on $\calM_d^N$, so we need to check that the only poles occur where $\rho =0$. The only way to get a pole is if one of the partial derivatives in the multiplier matrix has a pole. Thus, we are looking at the denominators of the partial derivatives of a dehomogenization of $f_{gen}$. Because the multiplier is independent of the dehomogenization choice, we could equally well dehomogenize at any of the coordinates $x_0,\ldots, x_N$.
        Fix a dehomogenization index $b$; then such a partial derivative can be expressed as a rational function, for $\phi_v = \frac{f_v(x_0,\ldots,x_{b-1},1,x_{b+1},\ldots,x_N)}
        {f_{b}(x_0,\ldots,x_{b-1},1,x_{b+1},\ldots,x_N)}$, as
        \begin{equation*}
              \frac{\partial \phi_v}{\partial x_w} = \frac{(\partial f_v/\partial x_w)\cdot f_{b} - (\partial f_{b}/\partial x_w)\cdot f_v}{f_{b}^2}.
        \end{equation*}
        In particular, a pole occurs at $\alpha$ for this dehomogenization when $f_b(\alpha)=0$.
         Since the multiplier is independent of the choice of homogenization, $\sigma_{i,j}^{(1)}$ is undefined exactly when $f_b(\alpha)$ vanishes for every $0 \leq b \leq N$ for some $\alpha$. However, this is the condition in which the denominator of $\sigma_{i,j}^{(1)}$ is some power of the resultant.


        The only point of interest in modifying the above proof for $\sigma_{i,j}^{\ast(n)}$ is having the points and their multipliers be $\Gal(\overline{K}/K)$ invariant sets. For the $n$-periodic points, we have a simple system of polynomial equations obtained from $f^n(P) = P$. However, for the formal $n$-periodic points, we have an inclusion-exclusion:
        \begin{equation*}
            \Phi^{\ast}_n(f) = \prod_{d \mid n} (\Phi_{d}(f))^{\mu(n/d)}.
        \end{equation*}
        Each $\Phi_d(f)$ on the right-hand side is a system of polynomial equations obtained from $f^d(P) = P$ and, hence, is $\Gal(\overline{K}/K)$ invariant. Hence, $\Phi^{\ast}_n(f)$ is $\Gal(\overline{K}/K)$ invariant and the rest of the proof follows similarly to the previous situation.

    \item For the second part, we have already seen that the $\sigma^{(n)}_{i,j}$ are conjugation invariant, so combining with the first part we have
        \begin{equation*}
            \sigma^{(n)}_{i,j} \in \Q[\Hom_d^N]^{\SL_{N+1}} = \Q[\calM_d^N].
        \end{equation*}
        The statement follows similarly for $\sigma_{i,j}^{\ast(n)}$.
    \end{enumerate}
    \end{proof}

\section{Generators and Relations Among the $\sigma_{i,j}$} \label{sect_gen}
    It is well known that in dimension $1$ the elementary symmetric functions of the multipliers are not all independent. Specifically, we have the relation (see Hutz-Tepper \cite{Hutz10} or Fujimura-Nishizawa \cite[Theorem 1]{Fujimura})
    \begin{equation*}
        (-1)^{d+1}\sigma_{d+1} + (-1)^{d-1}\sigma_{d-1} + (-1)^{d-2} 2\sigma_{d-2} + \cdots  - (d-1)\sigma_{1} +d=0.
    \end{equation*}
    Milnor \cite{Milnor} made specific use of this general relation for degree $2$ in noting that $\sigma_3 +2 = \sigma_1$. This relation can be obtained by expanding the classical relation between the multipliers \cite[Theorem 12.4]{Milnor3},
    \begin{equation} \label{eq_mult_relation}
        \sum_{i=1}^{d+1} \frac{1}{1-\lambda_i} = 1
    \end{equation}
    when $\lambda_i \neq 1$ for all $1 \leq i \leq d+1$, in terms of the elementary symmetric polynomials. We look at two sources of relations in this section:
    \begin{enumerate}
        \item relations obtained algebraically among the $\sigma_{i,j}$
        \item relations obtained from the generalization of equation \eqref{eq_mult_relation}.
    \end{enumerate}
    There are further cases of relations among the eigenvalues studied by Guillot in \cite{Guillot2} in the case of quadratic maps on $\P^2$, which we will not touch upon in this article.

    \begin{thm}
        Let $\sigma_{i,j}^{(n)}$ be defined as above. Then every $\sigma_{i,j}^{(n)}$ with $i > j$ is dependent on $\{\sigma_{1,j}^{(n)},\ldots,\sigma_{j,j}^{(n)}\}$.
    \end{thm}
    \begin{proof}
        We notate $\sigma_i$ for the elementary symmetric polynomials and $\sigma_{i,j}^{(n)}$ as defined in \eqref{eq_sigma}.

        From the definition of $\Sigma_n(f)$, we can write
        \begin{equation*}
            \Sigma_n(f) = \prod_{i=1}^{D_n} w - (t^N - \sigma_1(\boldsymbol \lambda_i)t^{N-1} + \sigma_2(\boldsymbol \lambda_i)t^{N-2} + \cdots + (-1)^N\sigma_N(\boldsymbol \lambda_i)),
        \end{equation*}
        where $\sigma_k$ are the elementary symmetric polynomials and $\boldsymbol \lambda_u = \{\lambda_{u,1},\ldots,\lambda_{u,N}\}$ are the eigenvalues of the multiplier matrix for a periodic point of period $n$. We can then write the $\sigma_{i,j}^{(n)}$ as combinations of the $\sigma_k(\boldsymbol \lambda_u)$. In particular, we have
        \begin{align*}
            \sigma_{i,j}^{(n)} &= \sum_{\text{subsets $(u_1,\ldots,u_i)$ of $\{1,\ldots,D_n\}$}} \left( \mathop{\sum_{k=1}^i }_{\sum a_k=j}\sigma_{a_1}(\boldsymbol\lambda_{u_1})\sigma_{a_2}(\boldsymbol\lambda_{u_2}) \cdots \sigma_{a_k}(\boldsymbol\lambda_{u_k})\right)\\
            &= \sum_{v \in \Part(j,i)} \binom{D_n-len(v)}{i-len(v)}\boldsymbol{\sigma}_v,
        \end{align*}
        where $\Part(j,i)$ is the set of partitions of the integer $j$ with $i$ parts (allowing $0$ as a part), $len(v)$ is the number of nonzero parts, and $\boldsymbol{\sigma}_v = \prod_{i \in v} \sigma_{v_i} (\boldsymbol \lambda_{u_{v_i}})$. Note that each term actually has $D_n$ terms in the product. The ones not listed are all $\sigma_0(\boldsymbol \lambda) = 1$ (corresponding to additional zeros in the full length $D_n$ partition of $j$).

        Fix $j$ and consider $\sigma_{1,j}^{(n)},\ldots,\sigma_{D_n,j}^{(n)}$. Notice that every one is the same degree, $j$, as polynomials in the eigenvalues of the multiplier matrices and that there are no new partitions that occur in $\sigma_{i,j}^{(n)}$ than have already occurred in $\{\sigma_{1,j}^{(n)},\ldots,\sigma_{j,j}^{(n)}\}$. Further, each $\sigma_{b,j}^{(n)}$ contains a partition not found in $\sigma_{a,j}^{(n)}$ for $a<b\leq j$ and, hence, $\{\sigma_{1,j}^{(n)},\ldots,\sigma_{j,j}^{(n)}\}$ are independent. Since the $\boldsymbol{\sigma}_k$ are fixed values and the binomial coefficient depends only on the length of the partition (which can be no larger than $j$), we can think of $\sigma_{i,j}^{(n)}$ as linear combinations of $j$ unknowns. These $j$ unknowns are
        \begin{equation*}
            z_k = \mathop{\sum_{v \in \Part(j,i)}}_{len(v)=k}\boldsymbol{\sigma}_v.
        \end{equation*}
        At most $j$ of these can be independent so at most $j$ of the $\sigma_{i,j}^{(n)}$ for $0 \leq i \leq D_n$ are independent. Since the first $j$, $\{\sigma_{1,j}^{(n)},\ldots,\sigma_{j,j}^{(n)}\}$ are independent, the remaining are dependent.
    \end{proof}
    Now we turn to the generalization by Ueda of Milnor's Rational Fixed Point Theorem \cite[Theorem 12.4]{Milnor3} derived from a generalization of the Cauchy integral formula. This generalization also appears in Abate \cite{Abate} and Guillot \cite{Guillot}. Fatou and Julia both made use of the relation between the multipliers of the fixed points. Milnor formalized the statement and made extensive use of the relation. We recall Ueda's statement in our notation.
    \begin{prop}[\textup{Ueda \cite[Theorem 4]{Ueda2}}]
        Let $f:\P^N(\C) \to \P^N(\C)$ be holomorphic of degree $d \geq 2$. Let $t$ be the indeterminant for the characteristic polynomials. We have the relation
        \begin{equation}
            \sum_{P \in \Fix(f)}\frac{\gamma_{f,P}(t)}{\gamma_{f,P}(1)} = \frac{t^{N+1} - d^{N+1}}{t-d} = \sum_{k=0}^N d^k t^{N-k}.
        \end{equation}\label{eq_ueda}
    \end{prop}
    \begin{rem}
        Equating the coefficients of $t$ on both sides yields $N+1$ relations, but only $N$ of these $N+1$ relations are independent: taking the sum of the coefficient relations is the same as putting in $t=1$, which is counting the fixed points.
%
    \end{rem}
    Interestingly, we can create alternate forms for these relations by differentiating with respect to $t$; see, for example, Ueda \cite[Corollary 5]{Ueda2}.

    It is tempting to try to convert Ueda's relations among the multipliers to relations among the $\sigma_{i,j}^{(n)}$. While it is possible to write Ueda's relations in terms of the $\sigma_{i,j}^{(n)}$, the specific form typically depends on $N$, $d$, and $k$ since we need to consider the partitions of $k$. The following Corollary is one case where the form does not depend on $N$, $d$, and $k$.
    \begin{cor}
        Let $f:\P^N \to \P^N$ be a morphism. We have the relation
        \begin{equation*}
            (D_1-1) + \sum_{k=1}^{ND_1} (-1)^{k+1} (\sigma_{D_1,k}^{(1)} - \sigma_{D_1-1,k}^{(1)}) = 0.
        \end{equation*}
    \end{cor}
    \begin{proof}
        If all the fixed points are distinct, we consider
        \begin{equation*}
            \sum_{i=1}^{D_1} \frac{1}{\prod_{j=1}^N (1- \lambda_{i,j})} = 1.
        \end{equation*}
        Finding a common denominator and then clearing the denominator, this equality becomes
        \begin{equation*}
            \sum_{k=1}^{D_1} \mathop{\prod_{j=1}^{D_1}}_{j \neq k} \prod_{i=1}^N (1-\lambda_{j,i}) - \prod_{j=1}^{D_1}\prod_{i=1}^N (1-\lambda_{j,i}) = 0.
        \end{equation*}
        The first term produces the $\sigma_{D_1-1,i}^{(1)}$, the symmetric functions on sets of $D_1-1$ variables (with the constant $D_1$ coming from the constant term). The second term produces the $\sigma_{D_1,i}^{(1)}$, the full symmetric functions, (with a constant term of $1$). The signs are determined by the number of $(-1)$'s in the product. This proves the relation when the fixed points are distinct.

        Since the set of maps with distinct fixed points is dense in $\Hom_d^N$ (it is the complement of the closed variety defined by the common vanishing of the $N \times N$ minors of the Jacobian matrix of the fixed point variety) it follows that the function
        \begin{equation*}
            (D_1-1) + \sum_{i=1}^{ND_1} (-1)^{i+1} (\sigma_{D_1,i}^{(1)} - \sigma_{D_1-1,i}^{(1)}) \in \C[\Hom_d^N]
        \end{equation*}
        is identically zero.
%
    \end{proof}

%
%

    While we have illustrated a few relations among the $\sigma_{i,j}^{(n)}$, it would be interesting to determine a minimal set of generators and full set of relations among the $\sigma_{i,j}^{(n)}$.

 \section[Computing]{Computing\footnote{Thanks to Carlos D'Andrea for helpful communication on this section.} }\label{sect_computing}

    With the goal of trying to use the $\sigma_{i,j}^{(n)}$ as coordinates in the moduli space $\calM_d^N$ we now turn to explicitly computing the $\sigma_{i,j}^{(n)}$ for a given map $f:\P^N \to \P^N$ or family of maps. For a given map in dimension 1, given enough time, we could compute a splitting field of $f^n(z) = z$ find each of the fixed points, compute their multipliers, and calculate the $\sigma_i$. However, when dealing with families of maps, the coefficients are functions of one or more parameters, this method becomes completely impractical. Fortunately we can use resultants to compute the $\sigma_i$ without actually computing the fixed points or their multipliers. The key is the Poisson product form of the resultant of two polynomials:
    \begin{equation*}
        \Res(F,G) = \prod_{F(z) = 0} G(z).
    \end{equation*}
    If we set $F = f(z)-z$ and $G = w-f'(z)$ for an indeterminant $w$, the resultant (which can be calculated just in terms of the coefficients of $F$ and $G$) is a polynomial in $w$ with the $\sigma_i$ as coefficients. We would like something similar for $f:\P^N \to \P^N$, specifically a way to compute the $\sigma_{i,j}^{(1)}$ that does not involve computing the fixed points nor their multiplier matrices. While the theory of resultants does not quite work (wrong number of equations and variables) we are able to use tools from elimination theory to perform these computations. This causes the computations to rely on the calculation of Groebner bases, which can be quite slow, but is effective for the families discussed in the article. We first prove the general elimination theory result.

    \begin{prop}\label{prop_elim_theory_I}
        Let $X = V(f_1(x),\ldots,f_m(x)) \subset \A^N$ be a zero dimensional variety defined by polynomials $f_1,\ldots,f_m$, where $x = (x_1,\ldots, x_N)$. Let $g(x,t) \in K[x][t]$ be a polynomial. Consider the ideal
        \begin{equation*}
            I = (f_1,\ldots,f_m, w - g) \subset K[x][w,t].
        \end{equation*}
        Let $B$ be a Groebner basis for $I$ under the lexicographic ordering with $x > w > t$. Then, the only polynomial in $B$ in the variables $(w,t)$ has as zeros the polynomial $g$ evaluated at the finitely many (algebraic) points of $X$.
    \end{prop}
    \begin{proof}
        Let $G(w,t)$ be the polynomial in $B$ in the variables $(w,t)$. This polynomial is in the ideal generated by $(f_1,\ldots, f_m, w-g)$, so after specializing $x$ to a common root $a$ of $f_1=\ldots=f_m=0$ (i.e., a point of $X$) we have that for some polynomial $A$
        \begin{equation*}
            G(w,t) = A(a,w,t)(w-g(a,t)).
        \end{equation*}
        In other words, $g(a,t)$ is a root of $G(w,t)$.

        In the other direction, the ``elimination-extension
        theorem'' (see \cite[Chapter 3]{CLO}) guarantees that every root of $G(w,t)$ extends to a root of the full system $f_1=\ldots = f_m = w-g = 0$, so that it comes from $g(a,t)$ for some $a \in X$.
    \end{proof}

    We can use Proposition \ref{prop_elim_theory_I} to (usually) compute the symmetric functions of the characteristic polynomials of the multipliers without actually computing the periodic points or the multipliers. When there are multiplicities involved (i.e., two periodic points have the same characteristic polynomial), the Groebner basis calculation loses this multiplicity information and, hence, does not exactly compute the $\sigma_{i,j}^{(1)}$.
%
    There are two ways around this issue. One is to introduce a deformation parameter to ``separate'' the values, take a Groebner basis of the new system, and specialize the deformation parameter to $0$. However, in practice some care needs to be taken in choosing how to deform so that the values do in fact separate. An alternative is a modification on computing Chow forms (or $U$-resultants). We adopt the later approach.
    \begin{prop}
        Let $X = V(f_1(x),\ldots,f_m(x)) \subset \A^N$ be a zero dimensional variety defined by polynomials $f_1,\ldots,f_m$, where $x = (x_1,\ldots, x_N)$. Let $g(x,t) \in K[x][t]$ be a polynomial. Consider the ideal
        \begin{equation*}
            I = (f_1,\ldots,f_m, u_0g + u_1x_1 + \cdots + u_Nx_N) \subset K[x][u,t],
        \end{equation*}
        where $u = (u_0,\ldots,u_N)$ are indeterminants.
        Let $B$ be a Groebner basis of $I$ under the lexicographic ordering with $x > u > t$. Then, the only polynomial in $B$ in the variables $(u,t)$ is of the form
        \begin{equation*}
            \prod_{a \in X} g(a,t)u_0 + a_1u_1 + \cdots a_Nu_N.
        \end{equation*}
    \end{prop}
    \begin{proof}
        Essentially the same as the Proposition \ref{prop_elim_theory_I}.
    \end{proof}

    The following algorithm computes the product of the characteristic polynomials of the multipliers of the fixed points. Roughly the algorithm computes the characteristic polynomials for the fixed points one affine chart at a time. The specialization step (Step \ref{step_spec}) avoids duplication of periodic points that are in multiple affine charts.
    \begin{algo} \label{Algo:min} \strut\vspace{-1ex} 
        Let $f: \P^N \to \P^N$ be a morphism.\\

        \textbf{Input:} $f$

        \textbf{Output:} $\Sigma_1(f)$
      \begin{enumerate}[1.]
        \item Let $X$ be the zero dimensional variety defining the fixed points.
        \item Set $\sigma = 1$.
        \item For each $j$ from $N$ to $0$ do:
        \begin{enumerate}[i.]
            \item Consider the $j$-th affine chart $f_j:\A^N \to \A^N$ in variables $x_1,\ldots,x_N$ and the fixed point variety $X_j$ of $f_j$.
            \item Compute $g$ as the characteristic polynomial of the jacobian matrix
            \begin{equation*}
                g(x,t) = charpoly\left(\frac{\partial f_j}{\partial (x_1,\ldots,x_N)}\right).
            \end{equation*}
            \item Specialize to $(x_1,\ldots, x_N) \mapsto (x_1,\ldots,x_j, 0, \ldots,0)$. \label{step_spec}
            \item Compute a lexicographic ($x > u > w > t$) Groebner basis of the (specialized) ideal.
            \begin{equation*}
                B=(X_j, u_0(w-g) + u_1x_1 + \cdots + u_Nx_N)
            \end{equation*}
            \item For $G$ the element of $B$ in the variables $(u,t)$, specialize to $u_0 = 1$ and $u_i =0$ for $1 \leq i \leq N$, call the specialization $\tilde{G}$. Set
                \begin{equation*}
                    \sigma = \sigma \cdot \tilde{G}.
                \end{equation*}
        \end{enumerate}
        \item Return $\sigma$
      \end{enumerate}
    \end{algo}
    Note that if the symbolic characteristic polynomial is a rational function in $x$, say $g= \frac{g_{num}}{g_{den}}$, then we can take the ideal:
    \begin{equation*}
        I = (X_f, u_0(wg_{den}-g_{num}) + g_{den}(u_1x_1 + \cdots + u_Nx_N)).
    \end{equation*}
    For $\Sigma_n(f)$, replace $f$ with $f^n$.

\begin{code}
\subsection{Code}

\begin{python}
def sigma(self,n, chow=False):
    d = self.degree()
    N = self.domain().dimension_relative()
    F2 = self.nth_iterate_map(n)
    P2 = F2.domain()
    T = 1
    f = F2
    for j in range(N,-1,-1):
        X = f.periodic_points(1,return_scheme=True, minimal=False)
        Xa = X.affine_patch(j)
        fa = f.dehomogenize(j)
        Pa = fa.domain()
        if chow:
            R = PolynomialRing(self.base_ring(),'v', N+N+3, order='lex')
            im = [R.gen(i) for i in range(j)] + (N-j)*[0]
            R_zero = {R.gen(N):1}
            for j in range(N+1,N+N+1):
                R_zero.update({R.gen(j):0})
            vars = list(R.gens())
            t = vars.pop()
            w = vars.pop()
            vars = vars[:N]
        else:
            R = PolynomialRing(self.base_ring(),'v', N+2, order='lex')
            im = list(R.gens())[:j] + (N-j)*[0]
            vars = list(R.gens())
            t = vars.pop()
            w = vars.pop()

        phi = Pa.coordinate_ring().hom(im,R)
        MS = MatrixSpace(R,N,N)
        M = t*MS.one()
        try:
            g = (M-jacobian([phi(F.numerator())/phi(F.denominator()) for F in fa], vars)).det()
            if chow:
                I = R.ideal([phi(h) for h in Xa.defining_polynomials()] + [w*g.denominator()-R.gen(N)*g.numerator() + sum(R.gen(j-1)*R.gen(N+j)*g.denominator() for j in range(1,N+1))])
            else:
                I = R.ideal([phi(h) for h in Xa.defining_polynomials()] + [w*g.denominator()-g.numerator()])
            G = I.groebner_basis()
            T = T*G[-1]
        except ZeroDivisionError:
            infin = self.domain()([1] + N*[0])
            if self(infin) == infin:
                T *= w-self.multiplier(infin,1).characteristic_polynomial(t)

        if chow:
            T2 = T.specialization(R_zero)
            newR = PolynomialRing(self.base_ring(), 'w,t',2, order='lex')
            psi = T2.parent().hom(N*[0]+list(newR.gens()), newR)
        else:
            T2 = T
            newR = PolynomialRing(self.base_ring(), 'w,t',2, order='lex')
            psi = T2.parent().hom(N*[0]+list(newR.gens()), newR)
    if chow:
        T = T.specialization(R_zero)
    return psi(T)
\end{python}
\end{code}

\section{McMullen's Theorem and Special Families}\label{sect_mcmullen}
    One of the main motivations of the current work is Milnor parameters and McMullen's Theorem. Milnor \cite{Milnor} proved that $M_2(\C) \cong \A^2(\C)$ and Silverman generalized this to $\Z$ \cite{Silverman9}. The isomorphism is explicitly given by the first two elementary symmetric polynomials of the multipliers of the fixed points. For the case of polynomials, Fujimura (and others) proved that the symmetric function in the multipliers of the fixed points gives a $(d-2)!$-to-1 mapping \cite{Fujimura,Sugiyama2}. Hutz-Tepper \cite{Hutz10} prove that for polynomials of degree $\leq 5$, adding the symmetric functions of the 2-periodic multipliers makes the mapping one-to-one and conjecture the same holds for polynomials of any degree. They also show for degree 3 rational functions that while the map to the fixed point multiplier symmetric functions is infinite-to-one, by adding the 2-periodic point multiplier symmetric functions the mapping is (generically) 12-to-one. The methods in all these cases are explicitly computational.

    Using complex analytic methods, McMullen proved in dimension $1$ that by including symmetric functions of the multipliers of the periodic points of enough periods, the multiplier map will always be finite-to-one away from the locus of Latt\`es maps \cite{McMullen2}. The Latt\`es maps must be avoided since they all have the same set of multipliers. We proposed the following generalization of McMullen's Theorem
    \begin{conj}\label{conj_mcmullen}
        Let $f:\P^N \to \P^N$ be a morphism. Define the map
        \begin{equation*}
            \tau_{d,n}^N: \calM_d^N \to \A^M
        \end{equation*}
        by
        \begin{equation*}
            \tau_{d,n}^N(f) = (\boldsymbol \sigma^{(1)}, \ldots, \boldsymbol \sigma^{(n)}),
        \end{equation*}
        where $\boldsymbol \sigma^{(k)}$ is the complete set of sigma invariants for the points of period $k$.
        Then for large enough $n$, $\tau_{d,n}^N$ is quasi-finite on a Zariski open set.
    \end{conj}
    We prove a few special cases and describe a few special subvarieties where the map $\tau_{d,n}^N$ is constant for all $n$.

\subsection{Isospectral Families}
    \begin{defn} \label{defn_isospectral}
        We say that two maps $f,g:\P^N \to \P^N$ are \emph{isospectral} if they have the same image under $\tau_{d,n}^N$ for all $n$.
        Similarly, we say that a family $f_a:\P^N \to \P^N$ is \emph{isospectral} if its image under $\tau_{d,n}^N$ is a point for all $n$.
    \end{defn}

\subsubsection{Latt\`es}
One way to generate isospectral maps in higher dimensions is to apply a construction to a family of Latt\`es maps. For example, symmetrization \cite{Hutz18}, cartesian products, and Segre embeddings can be used to construct isotrivial families starting with a Latt\`es family.

It is worth mentioning that Rong \cite[Theorem 4.2]{Rong} proves that symmetric maps (up to semi-conjugacy) are the only Latt\`es on $\P^2$.

\begin{thm}\label{iso_symmetric}
    Let $f_a:\P^1 \to \P^1$ be a family of Latt\`es maps of degree $d$. Then the $k$-symmetric product $F$ is an isospectral family in $\Hom_d^k$.
\end{thm}
\begin{proof}
    The multipliers of the symmetric product $F$ depend only on the multipliers of $f$ \cite{Hutz18}.
\end{proof}

\begin{exmp}
    We compute with an example from Gauthier-Hutz-Kaschner \cite{Hutz18}. Starting with the Latt\`es family
    \begin{align*}
        f_a:\P^1 &\to \P^1\\
        [u,v] &\mapsto [(u^2 - av^2)^2 , 4uv(u-v)(u - av)]
    \end{align*}
    we compute the $2$-symmetric product
    \begin{align*}
        F_a&: \P^2 \to \P^2\\
        [x,&y,z] \mapsto [((x+az)^2 - ay^2)^2,\\
         &4((x + az)^3y + 2(a + 1)(x + az)^2xz + a(x + az)y^3 - 8axyz(x + az) - (a + 1)(x^2y^2 + a^2y^2z^2)),\\
         &16xz(x-y+z)(x-ay+a^2z)].
    \end{align*}
    We compute the product of $(w-\gamma(t))$, without multiplicity, to have
    \begin{align*}
        \Sigma_n(f) =& w^5 + w^4(-5t^2 -4t) + w^3(10t^4 + 16t^3 - 4t^2 -32t - 48)\\
        &+ w^2(-10t^6 - 24t^5 + 12t^4 + 112t^3 + 240t^2 + 256t + 128)\\
        &+ w(5t^8 + 16t^7 - 12t^6 - 128t^5 - 336t^4 - 576t^3 - 576t^2 - 256t)\\
        &- t^{10} -4t^9 + 4t^8 + 48t^7 + 144t^6 + 320t^5 + 448t^4 + 256t^3.
    \end{align*}
    Note that this does not depend on the parameter $a$.
\begin{code}
\begin{python}
#symmetric lattes family from Gauthier-Hutz-Kaschner
#should have sigmas independent of t
sage: R.<u> = QQ[]
sage: P.<x,y,z>=ProjectiveSpace(FractionField(R),2)
sage: f=DynamicalSystem([((x+u*z)^2 - u*y^2)^2, 4*((x + u*z)^3*y + 2*(u + 1)*(x + u*z)^2*x*z + u*(x + u*z)*y^3 - 8*u*x*y*z*(x + u*z) - (u + 1)*(x^2*y^2 + 0*u*x^2*z^2 + u^2*y^2*z^2)), 4*4*x*z*(x-y+z)*(x-u*y+u^2*z)])
sage: s=sigma(f,1)

Wall time: 73933.57 s, ~36Gb memory
\end{python}
\end{code}
\end{exmp}


Next we consider cartesian products of maps. Given morphisms $f:\P^N \to \P^N$ and $g:\P^M \to \P^M$, we define a map $h = f \times g:\P^{N+M+1} \to \P^{N+M+1}$ as the induced map by the coordinates of $f$ and $g$. Specifically,
\begin{align*}
    h(x_0,\ldots,x_N,x_{N+1},\ldots,x_{N+M+1}) &= [f_0(x_0,\ldots,x_N),\ldots,f_N(x_0,\ldots,x_N),\\
      &g_0(x_{N+1},\ldots,x_{N+M+1}), \ldots, g_M(x_{N+1},\ldots,x_{N+M+1})].
\end{align*}
For this product to be well defined as a projective map, we must have $\deg(f) = \deg(g)$. Note that the resulting map $h$ is a morphism.

\begin{lem}\label{lem_cart_fixed_pt}
    Let $f:\P^N \to \P^N$ and $g:\P^M \to \P^M$ be morphisms of degree $d > 1$ and $h = f \times g$. The fixed points of $h$ are of the following three forms.
    \begin{enumerate}
        \item (fixed point of $f)\cdot k\cdot (d-1$ root of unity) $\times$ (fixed point of $g$) where $k$ satisfies $f(kx_0,\ldots, kx_N)=(k\alpha x_0,\ldots, k\alpha x_N)$, where $g(x_{N+1},\ldots,x_{N+M+1}) = (\alpha  x_{N+1},\ldots, \alpha  x_{N+M+1})$ for some nonzero constant $\alpha$
        \item $[0,\ldots,0] \times$ (fixed point of $g$)
        \item (fixed point of $f$)$ \times [0,\ldots,0]$
    \end{enumerate}
\end{lem}
\begin{proof}
    Recall that for a morphism of degree $d$ on $\P^N$ there are $\frac{d^{N+1}-1}{d-1} = d^N + \cdots + d + 1$ fixed points (counted with multiplicity). Each of the points enuemrated in the statement is clearly fixed and distinct from each other 
    and there are a total of
    \begin{align*}
        (d-1)\cdot \frac{d^{N+1}-1}{d-1}\cdot \frac{d^{M+1}-1}{d-1} + \frac{d^{N+1}-1}{d-1} + \frac{d^{M+1}-1}{d-1} &= d^{N+1} \frac{d^{M+1}-1}{d-1} + \frac{d^{N+1}-1}{d-1}\\
        &= d^{N+M+1} + \cdots + d+1
    \end{align*}
    of these points (counted with multiplicity). Therefore, these are all of the fixed points.
\end{proof}


\begin{thm}\label{thm_cart_prod}
    Let $f_a:\P^N \to \P^N$ and $g_b:\P^M \to \P^M$ be isospectral families of morphisms with $\deg(f_a) = \deg(g_b)$. Then the cartesian product family $h_{a,b} = f_a \times g_b$ is isospectral in $\Hom_d^{N+M+1}$.
\end{thm}
\begin{proof}
    We need to show that the eigenvalues of the multiplier matrices of $h_{a,b}$ depend only on the eigenvalues of the multiplier matrices of $f_a$ and $g_b$. We consider each type of fixed point of $h_{a,b}$ from Lemma \ref{lem_cart_fixed_pt} in turn.
    \begin{itemize}
        \item For a fixed point $Q$ of the form (fixed point of $f)\cdot k\cdot (d-1$ root of unity) $\times$ (fixed point of $g$), where $k$ satisfies $f(kx_0,\ldots, kx_N)=(k\alpha x_0,\ldots, k\alpha x_N)$ and where $g(x_{N+1},\ldots,x_{N+M+1}) = (\alpha  x_{N+1},\ldots, \alpha  x_{N+M+1})$, notate
            \begin{equation*}
                Q = \alpha \zeta Q_f \times Q_g,
            \end{equation*}
            where $\zeta$ is the $d-1$st root of unity.
            We will show that the eigenvalues of the multiplier matrix of $Q$ are the eigenvalues of the multiplier matrix of $Q_f$, the eigenvalues of the multiplier matrix of $Q_g$, and $d$.

            At least one coordinate of $Q_f$ is nonzero, if we dehomogenize $h_{a,b}$ at that coordinate and compute the multiplier matrix, we have a matrix of the form
            \begin{equation*}
                \begin{pmatrix}m_{f,Q_f} & 0\\-- & G\end{pmatrix}
            \end{equation*}
            so the eigenvalues of the multiplier matrix of $Q_f$ for $f$ are eigenvalues of this matrix. Similarly at least one coordinate of $Q_g$ is nonzero and we see that the eigenvalues of the multiplier matrix of $Q_g$ are also eigenvalues.

            There is one remaining undetermined eigenvalue, which we now show is $d$, the degree of $h_{a,b}$. Let $i$ be the coordinate of $Q_g$ that is nonzero. Dehomogenizing at $i$ and, with a slight abuse of notation, labeling the new coordinates $(x_0,\ldots,x_N,x_{N+1},\ldots, \hat{x_i},\ldots,x_{n+M+1})$ and computing the multiplier matrix, we get
            \begin{equation*}
            \begin{pmatrix}
                \frac{\partial f_j/g_i}{\partial (x_0,\ldots,x_N)} & 0\\
                -- & m_{G,Q_g}
            \end{pmatrix}
            \end{equation*}
            Where the upper left hand block is the Jacobian matrix of the dehomogenization with respect to the first $N+1$ variables,
            \begin{equation*}
                J_f = \begin{pmatrix}
                    \frac{g_i\frac{\partial f_0}{\partial x_0}}{g_i^2} & \cdots & \frac{g_i\frac{\partial f_0}{\partial x_N}}{g_i^2} \\
                    \vdots&\cdots & \vdots\\
                    \frac{g_i\frac{\partial f_N}{\partial x_0}}{g_i^2} & \cdots & \frac{g_i\frac{\partial f_N}{\partial x_N}}{g_i^2}
                \end{pmatrix} = \frac{1}{g_i}\begin{pmatrix}
                    \frac{\partial f_0}{\partial x_0} & \cdots & \frac{\partial f_0}{\partial x_N}\\
                    \vdots&\cdots & \vdots\\
                    \frac{\partial f_N}{\partial x_0} & \cdots & \frac{\partial f_N}{\partial x_N}
                \end{pmatrix}
            \end{equation*}
            We are looking for roots of the characteristic polynomial $\det(J_f - tId)$ or, equivalently, that the matrix $J_f - dId$ is singular. We will see that its columns are dependent using Euler's identity for homogeneous polynomials
            \begin{equation*}
                df_j = \sum_{l=0}^{N} x_l \frac{\partial f_j}{\partial x_l}.
            \end{equation*}
            Taking the linear combination of the columns with $(x_0,\ldots,x_N)$ we arrive at the column vector
            \begin{equation*}
                \begin{pmatrix}
                \left(\sum_{l=0}^N x_l\frac{\partial f_0}{\partial x_l}\right) - x_0 d\\
                \vdots\\
                \left(\sum_{l=0}^N x_l\frac{\partial f_N}{\partial x_l}\right) - x_N d
                \end{pmatrix}.
            \end{equation*}
            Notice that $f_i(x_0,\ldots, x_N) = x_i$ since we are working with a fixed point (the factor of $\alpha$ cancels since we have dehomogenized). Therefore, by Euler's identity, this is the zero column vector and the columns are dependent. Hence, the matrix $(J_f - dId)$ is singular and $d$ is an eigenvalue of $J_f$.

%
%

        \item For the remaining two fixed point forms, we dehomogenize at a nonzero coordinate (say either $f_i$ or $g_i$). Then the multiplier matrix is of the form
            \begin{equation*}
            \begin{pmatrix}
                0 & 0\\
                -- & m_{G,Q_g}
            \end{pmatrix},
            \end{equation*}
            so we have eigenvalues $0$ and $\lambda_{Q_g}$ (or $0$ and $\lambda_{Q_f}$, respectively).
    \end{itemize}

    The eigenvalues of the multiplier matrices of $h_{a,b}$ depend only on the degree, the dimension, and the eigenvalues of the multiplier matrices of $f_a$ and $g_b$. Since $f_a$ and $g_b$ are isospectral, then so is $h_{a,b}$.
\end{proof}

%
%

\begin{exmp}
    Consider the two maps
    \begin{align*}
        f_a:\P^1 &\to \P^1 \quad [u,v] \mapsto [(u^2 - av^2)^2 , 4uv(u-v)(u - av)]\\
        g:\P^1 &\to \P^1 \quad
        [z,w] \mapsto [z^4,w^4]
    \end{align*}
    The cartesian product is the family
    \begin{align*}
        F_a:\P^3 &\to \P^3\\
        [(u^2 &- av^2)^2 , 4uv(u-v)(u - av) , z^4 , w^4].
    \end{align*}
    We compute the $\sigma_{i,j}^{(1)}$ (in Sage) and see that the result does not depend on the parameter $a$.
    \begin{align*}
        \Sigma_1(f) = &w^{11} + w^{10}(-11t^3 + 40t^2 -64t + 32)\\
        &+ w^9(55t^6 -400t^5 + 1272t^4 -2080t^3 + 1280t^2 + 1024t -2048)\\
        &+ w^8(-165t^9 + 1800t^8 -8568t^7 + 22240t^6 -28864t^5 -128t^4 + 59392t^3 -75776t^2 + 32768t)\\
        &+ \cdots
    \end{align*}

\begin{code}
\begin{python}
sage: R.<a>=QQ[]
sage: P.<u,v,z,w>=ProjectiveSpace(FractionField(R),3)
sage: f=DynamicalSystem([(u^2 - a*v^2)^2 , 4*u*v*(u-v)*(u - a*v), z^4,w^4])
sage: sigma(f,1)
\end{python}
\end{code}
\end{exmp}

Note that we can create similar families for $\P^2$.
\begin{exmp}
    Consider the Latt\`es family given by multiplication by $2$ on $y^2 = x^3 + a$:
    \begin{align*}
        f_a:\P^1 \to \P^1 \quad [u,v] \mapsto [u^4 + (-8a)uv^3 : 4u^3v + 4av^4]
    \end{align*}
    and the family
    \begin{align*}
        F_a:\P^2 &\to \P^2\\
        [ x^4 + &(-8a)xy^3 : 4x^3y + 4ay^4 : z^4 ].
    \end{align*}
    We compute the $\sigma_{i,j}^{(1)}$ (in Sage) without multiplicity and see that the result does not depend on the parameter $a$.
    \begin{align*}
        \Sigma_1(f) = &w^5 + w^4(-5t^2 + 8t -8) + w^3(10t^4 -32t^3 + 28t^2 + 48t -128)\\
        &+ w^2(-10t^6 + 48t^5 -36t^4 -176t^3 + 320t^2 + 256t) + w(5t^8 -32t^7\\
        &+ 20t^6 + 208t^5 -256t^4 -512t^3) - t^{10} + 8t^9 + (-4)t^8 + (-80)t^7 + 64t^6 + 256t^5
    \end{align*}

\begin{code}
\begin{python}
sage: R.<a>=QQ[]
sage: P1.<u,v>=ProjectiveSpace(R,1)
sage: F = P1.Lattes_map(EllipticCurve(R,[0,a]),2)
sage: P2.<x,y,z>=ProjectiveSpace(FractionField(R),2)
sage: S=P2.coordinate_ring()
sage: g=DynamicalSystem([x^4 + (-8*a)*x*y^3 , 4*x^3*y + 4*a*y^4,z^4])
sage: sigma(g,1)
\end{python}
\end{code}
\end{exmp}

Next we consider products of maps (of the same degree) embedded into $\P^N$ by the Segre embedding. Given morphisms $f:\P^N \to \P^N$ and $g:\P^M \to \P^M$, we define a map $f \times g:\P^{N} \times \P^M \to \P^{N} \times \P^M$. Via the Segre embedding, this product induces a map $h:\P^{(N+1)(M+1)-1} \to \P^{(N+1)(M+1)-1}$.

\begin{thm} \label{iso_segre}
    Let $f_a:\P^N \to \P^N$ be an isospectral family of degree $d\geq 2$ and $g:\P^M \to \P^M$ the degree $d$ powering map. Then the family of endomorphisms of $h_{a}:\P^{(N+1)(M+1)-1} \to \P^{(N+1)(M+1)-1}$ induced by the Segre embedding of $f_a \times g$ is isospectral in $\Hom_d^{(N+1)(M+1)-1}$.
\end{thm}

\begin{proof}
    The map induced by the Segre embedding of $f_a$ and the powering map is (after permuting coordinates) the cartesian product of $M+1$ copies of $f_a$. Applying Theorem \ref{thm_cart_prod} inductively to the product gives the result.
\end{proof}

\begin{exmp}
    Consider the Latt\`es map induced by multiplication by 2 on the Mordell family: $y^2 = x^3 + a$,
    \begin{align*}
        F_a:\P^1 &\to \P^1\\
        (x : y) &\mapsto (x^4 + (-8a)xy^3 : 4(x^3y + ay^4))
    \end{align*}
    and the powering map
    \begin{align*}
        G:\P^1 &\to \P^1\\
        (x : y) &\mapsto (x^4 : y^4).
    \end{align*}

    We compute the map induced on $\P^3 \to \P^3$ by the Segre embedding given by
    \begin{align*}
        f_a:\P^3 &\to \P^3\\
        (u_0 : u_1 : u_2 : u_2) &\mapsto (u_0^4 + (-8a)u_0u_2^3 : u_1^4 + (-8a)u_1u_3^3 :\\
        &4(u_0^3u_2 + au_2^4) : 4(u_1^3u_3 + au_3^4)).
    \end{align*}
    We compute the $\sigma_{i,j}^{(1)}$ (in Sage) without multiplicity and see that they do not depend on $a$
    \begin{align*}
        \Sigma_1(f) = &w^8 + w^7(-8t^3 + 8t^2 + 20t -16)\\
        &+ w^6(28t^6 -56t^5 -140t^4 + 256t^3 + 160t^2 -512t -512)\\
        &+ w^5(-56t^9 + 168t^8 + 420t^7 -1280t^6 -1056t^5 + 4032t^4 + 3072t^3 -4096t^2 -4096t)\\
        &+ w^4(70t^{12} -280t^{11} -700t^{10} + 3120t^9 + 2864t^8 -13888t^7 -9856t^6 + 26624t^5 + 26624t^4)\\
        &+ \cdots
    \end{align*}
\begin{code}
\begin{python}
sage: S1.<a>=QQ[]
sage: S=FractionField(S1)
sage: P1.<x,y>=ProjectiveSpace(S,1)
sage: F = P1.Lattes_map(EllipticCurve(S,[0,a]),2)
sage: G = DynamicalSystem([x^4,y^4])
sage: F,G
sage: P3.<x0,x1,x2,x3>=ProjectiveSpace(S,3)
sage: phi=P1.coordinate_ring().hom([x0,x1],P3.coordinate_ring())
sage: phi2=P1.coordinate_ring().hom([x2,x3],P3.coordinate_ring())
sage: R=P3.coordinate_ring()
sage: f=DynamicalSystem([phi(F[0]),phi(F[1]),phi2(G[0]),phi2(G[1])])
sage: s11=(P1*P1).segre_embedding()

sage: R.<x0,x1,x2,x3,u0,u1,u2,u3>=PolynomialRing(S,order='lex')
sage: I=R.ideal([u0-x0*x2, u1-x0*x3, u2-x1*x2,u3-x1*x3, f[0]*f[2]])
sage: print I.groebner_basis()[8]
u0^4 + (-8*a)*u0*u2^3
sage: I=R.ideal([u0-x0*x2, u1-x0*x3, u2-x1*x2,u3-x1*x3, f[0]*f[3]])
sage: print I.groebner_basis()[13]
u1^4 + (-8*a)*u1*u3^3
sage: I=R.ideal([u0-x0*x2, u1-x0*x3, u2-x1*x2,u3-x1*x3, f[1]*f[2]])
sage: print I.groebner_basis()[9]
u0^3*u2 + a*u2^4
sage: I=R.ideal([u0-x0*x2, u1-x0*x3, u2-x1*x2,u3-x1*x3, f[1]*f[3]])
sage: print I.groebner_basis()[15]
u1^3*u3 + a*u3^4

sage: S1.<a>=QQ[]
sage: P3.<u0,u1,u2,u3>=ProjectiveSpace(FractionField(S1),3)
sage: f0=(u0^4 + (-8*a)*u0*u2^3)
sage: f1=(u1^4 + (-8*a)*u1*u3^3)
sage: f2=4*(u0^3*u2 + a*u2^4)
sage: f3=4*(u1^3*u3 + a*u3^4)
sage: f=DynamicalSystem([f0,f1,f2,f3])
sage: s11(s11.domain()([1,1,2,1])), s11(s11.domain()(list(F([1,1])) + list(G([2,1])))), f([2,1,2,1])
sage: sigma(f,1)
\end{python}
\end{code}
\end{exmp}

It is not clear if endomorphisms of $\P^N$ induced by the Segre embedding applied to more general isospectral families are still isospectral. The question comes down to the fixed points of the induced map that are not in the image of the Segre embedding. It seems possible that the multipliers of these points could depend on the parameter. Surprisingly, the few examples attempted by computation appeared to be isospectral, but the full computation was beyond the reach of the machine being used.


\subsection{Finite-to-One}
In this section, we prove the multiplier map $\tau_{d,n}^N$ is finite-to-one for certain special families.

\subsubsection{Split Polynomial Endomorphisms}
    We first treat the simplest case, split polynomial endomorphisms. On affine space, a split polynomial endomorphism is an endomorphism where each coordinate is a single variable polynomial: $F = (F_1(x_1), F_2(x_2),\ldots, F_N(x_N)):\A^N \to \A^N$, for polynomials $F_1,\ldots,F_N$. A projective split polynomial endomorphism is the homogenization of an affine split polynomial endomorphism. See Ghioca-Nguyen \cite{ghioca4} for a study of periodic subvarieties for split polynomial maps and the Dynamical Mordell-Lang Conjecture in the disintegrated (not Chebyshev or power map) case. We must have $\deg(F_1)= \cdots = \deg(F_N)$ for the resulting projective map to be a morphism.

    The following two lemmas are simple calculations.
    \begin{lem}\label{lem_split_multipliers}
        The multiplier matrix of an affine fixed point of a split polynomial endomorphism is diagonal with entries the multipliers of the fixed points of the coordinate polynomials as maps of $\A^1$.
    \end{lem}

    \begin{lem} \label{lem_split_permutations}
        Let $f:\P^N \to \P^N$ be a split polynomial endomorphism. Let $m = \begin{pmatrix} m'&0\\0&1
        \end{pmatrix}$ for a $(N-1)\times(N-1)$ permutation matrix $m'$. Then the conjugation $f^{m}$ is obtained by permuting the first $N$ coordinate functions of $f$ by $m'$.
    \end{lem}

    \begin{thm} \label{finite_split}
        The fixed point multiplier map, $\tau_{d,1}^N$, is generically $((d-2)!)^{N}$-to-one when restricted to split polynomial endomorphisms of degree $d$ on $\P^N$.
    \end{thm}
    \begin{proof}
        To check that $\tau_{d,1}^N$ is finite-to-one and compute the degree we assume we are given $\Sigma_1(f)$ for some split polynomial endomorphism $f$.

        The $\sigma_{i,j}^{(1)}$ are the coefficients of the polynomial $\Sigma_1(f)$, which we can factor to get (unordered) sets of eigenvalues of the multiplier matrices. We need to determine how many ways we can split these unordered eigenvalues into multiplier spectra for the coordinate maps.

        After a change of variables, we can assume that each $F_i$ is monic. The affine fixed points are all possible cartesian products of the fixed points of $F_1,\ldots,F_N:\A^1 \to \A^1$. The fixed points at infinity are $(x_1 : \dots : x_N : 0)$ with each $x_i \in \{-1,0,1\}$, not all $0$. In particular, we know which eigenvalues come from the fixed points at infinity. For the affine fixed points, since the multiplier matrices are diagonal (Lemma \ref{lem_split_multipliers}) the multiplier of each fixed point of each coordinate function is repeated a specific number of times: $(d+1)^{N-1}$. Further, based on which eigenvalues occur in which (unordered) sets, we can split the eigenvalues into multiplier spectra for each coordinate function in only one way. Consequently, the only freedom of choice we have is to permute the coordinate functions. 

        Having partitioned the eigenvalues into multiplier spectra based on coordinate functions, we can apply Fujimura's results summarized in  \cite{Fujimura2} that any given set of multiplier invariants arising from the multiplier spectra of the fixed points of a polynomial corresponds to $(d-2)!$ possible conjugacy classes of polynomials. There are $N$ sets of multiplier spectra, so we have (generically) $((d-2)!)^N$ possible sets of $N$ coordinate functions. Then these $N$ coordinate functions can be arranged in any permutation; however, all of these $N!$ permutations are conjugate (Lemma \ref{lem_split_permutations}), so the total degree of $\tau_{d,1}^N$ restricted to split polynomial endomorphisms is $((d-2)!)^N$.
    \end{proof}

    \begin{rem}
        Moreoever, we can apply Sugiyama's algorithm \cite{Sugiyama2} to determine which sets of multipliers $\sigma_{i,j}^{(1)}$ are in the image of $\tau_{d,1}^N$ and what the degree of a specific fiber is by applying the algorithm componentwise.
    \end{rem}

    Hutz-Tepper conjectured in \cite{Hutz10} that $\tau_{d,2}^1$ is injective for polynomial maps in dimension 1, i.e., including the multipliers of the 2-periodic points. That conjecture implies a similar statement here: $\tau_{d,n}^N$ for $n \geq 2$ is (generically) one-to-one when restricted to split polynomial endomorphisms.

\begin{exmp}
    For degree $2$ polynomials
    \begin{equation*}
        F = (x^2 + c, y^2 + d),
    \end{equation*}
    we have that the invariants $\{\sigma_{i,j}^{(1)}\}$ are generated by
    \begin{align*}
        \sigma_{2,2}^{(1)} &= 8(c + d) + 60\\
        \sigma_{2,3}^{(1)} &= 16(c + d) + 24.
    \end{align*}
    In particular, the pair $(c,d)$ are determined up to permutation. Recall from Milnor that the family of quadratic polynomials $f(x) = x^2+c$ is the line $(\sigma_1,\sigma_2) = (2, 4c)$ in $\calM_2^1$, so that each pair $(c,d)$ corresponds to exactly one function $F$. Since the permutations are conjugate, the multiplier mapping $\tau_{2,1}^2$ is (generically) $1$-to-$1$.
\end{exmp}

\subsubsection{Triangular polynomials}
\begin{defn}
    A \emph{triangular polynomial} is a map of the form
    \begin{align*}
        F:\A^N &\to \A^N\\
        (x_1,\ldots,x_N) &\mapsto (F_1(x_1),F_2(x_1,x_2),\ldots, F_N(x_1,\ldots,x_N))
    \end{align*}
    for polynomials $F_1,\ldots,F_N$.
\end{defn}
We are specifically interested in the case when the homogenization is an endomorphism of $\P^N$ and call such maps \emph{triangular polynomial endomorphisms}. For the homogenization to be an endomorphism, it is necessary that $\deg(F_1) = \cdots = \deg(F_N)$.

%
%
%

The following combinatorial lemma is needed to ensure we have enough equations to determine our map through interpolation.
\begin{lem} \label{lem_interpolation_count}
    Fix a positive integer $d \geq 2$. Then, for every positive integer $n \geq 1$ we have
    \begin{equation*}
        \binom{d+n}{d} \leq \sum_{i=0}^{n}d^i.
    \end{equation*}
\end{lem}
\begin{proof}
    We proceed by induction on $n$. For $n=1$, we compute
    \begin{equation*}
        d+1 = \binom{d+1}{d} = d+1.
    \end{equation*}
    Now we induct on $n$ so that $n$ is at least $2$, and we compute
    \begin{align*}
        \binom{d+(n+1)}{d} &= \frac{(n+2)\cdots (n+d+1)}{d!}
        = (n+1)\frac{(n+2)\cdots (n+d)}{d!} + d\frac{(n+2)\cdots (n+d)}{d!}\\
        &= \frac{(n+1)(n+2)\cdots (n+d)}{d!} + \frac{d}{n+1}\frac{(n+1)(n+2)\cdots (n+d)}{d!}\\
        &\leq (d^n + \cdots + d + 1) + \frac{d}{n+1}(d^n + \cdots + d + 1).
    \end{align*}
    We need
    \begin{equation*}
        \frac{1}{n+1}(d^{n+1} + \cdots + d) \leq d^{n+1},
    \end{equation*}
    which is the same as
    \begin{equation*}
        d^n + \cdots + d + 1 \leq nd^{n+1}.
    \end{equation*}
    Since $n$ is at least $2$, it is sufficient to show that
    \begin{equation*}
        d^{n-1} + \cdots + d + 1 \leq d^n.
    \end{equation*}
    Since $\frac{d^n-1}{d-1} = d^{n-1} + \cdots + d + 1$ and $d \geq 2$, this is clear.
\end{proof}

\begin{thm} \label{thm_triangular}
    The fixed point multiplier map $\tau_{d,1}^N$ is (generically) finite-to-one when restricted to triangular polynomial endomorphisms.
\end{thm}
\begin{proof}

    Let $f$ be the homogenization of the triangular polynomial endomorphism $(F_1,\ldots,F_N):\A^N \to \A^N$.
    To check that $\tau_{d,1}^N$ is finite-to-one assume we are given $\Sigma_1(f)$ for some triangular polynomial endomorphism $f$.
    The $\sigma_{i,j}^{(1)}$ are the coefficients of the polynomial $\Sigma_1(f)$, which we can factor to get (unordered) sets of eigenvalues of the multiplier matrices.

    From this finite set of eigenvalues there are a finite number of subsets which could be the (affine) multipliers of fixed points for $F_1$ as an endomorphism of $\P^1$. Since $F_1$ is a polynomial, the fixed point at infinity is totally ramified and has multiplier $0$. Applying Fujimura \cite{Fujimura2}, there are finitely many possibilities for $F_1$ for each subset of eigenvalues (specifically, $(d-2)!$). Write each of these possibilities  in a monic centered form and compute the fixed points. For each fixed point $z$ of $F_1$ we can consider the single variable polynomial $F_2(z,x_2):\A^1 \to \A^1$. Again there are finitely many subsets of the eigenvalues that could be the multipliers of the fixed points of $F_2(z,x_2)$, so there are finitely many possibilities for $F_2(z,x_2)$. Again write each of these in a monic centered form in terms of the variable $x_2$. We now have essentially an interpolation problem. We know the values of the coefficients of the polynomial $F_2$ specialized at the fixed points of $F_1$. There are $d+1$ fixed points of $F_1$ and each coefficient is a polynomial of degree at most $d$ in $x_1$. Hence, the number of coefficient polynomials is no larger than the number of fixed points, so we can find the unique polynomial $F_2$.


    Repeating this process for each $F_3,\ldots,F_N$, we expect that there are finitely many possibilities for $F$. Each step is a multivariate Lagrange interpolation problem. Having distinct fixed points is an open condition in $\calM_d^N$ so Lemma \ref{lem_interpolation_count} shows that generically there are ``enough'' fixed points at each stage to have enough equations to uniquely determine the coefficient polynomials.
    The question remains as to whether enough of those equations are independent. The equations to determine each $F_k$ are linear equations in the images of the Veronese embeddings $\nu_j$ for  $1 \leq j \leq d$ applied to the fixed points of $(F_1,F_2,\ldots,F_{k-1})$ (as an endomorphism of $\P^{k-1}$). So, if we are in a situation of infinitely many solutions, then for some $j$, the images of these fixed points under $\nu_j$ satisfy a hyperplane equation in the Veronese variety. In particular, the fixed points of $(F_1,F_2,\ldots,F_{k-1})$ as an endomorphism of $\P^{k-1}$ are on a degree $j$ hypersurface in $\P^{k-1}$.
    We need to see that having such a dependency is a closed condition, so it does not happen generically. Consider the product spaces $\calM_d^N \times M(N,j)$ for $1 \leq j \leq d$ where $M(N,j)$ is the moduli space of degree $j$ hypersurfaces in $\P^N$. We are considering the conditions on pairs $(f, S)$ where $\Fix(f) \subset S$. Checking whether a subvariety is contained in a hypersurface gives a closed condition. Each of these closed conditions ($1 \leq j \leq d$) projects to a closed condition on $\calM_d^N$, so to have any such relation is a closed condition on $\calM_d^N$. Therefore, as long as there exists at least one map outside of this condition, then generically there is no such hypersurface relation among the fixed points.
    The powering map of degree $d$ has fixed points with each coefficient either $0$ or a $(d-1)$-st root of unity (but not all $0$). In particular, every possible point whose coordinates are $1$ or $0$ is fixed. These cannot all satisfy a single polynomial equation. Hence, generically, the fixed points are not all on a hypersurface of degree at most $d$ and the interpolation problem gives a unqiue solution.
\end{proof}

%

If we can effectively determine a single variable polynomial from the multipliers of its fixed points, the proof of Theorem \ref{thm_triangular} can be used to effectively compute the fibers of $\tau_1$. Recall that for a degree $2$ polynomial whose affine fixed points have multipliers $\lambda_1, \lambda_2$, we have multiplier invariants
\begin{equation*}
    (\sigma_1, \sigma_2) = (\lambda_1 + \lambda_2, \lambda_1\lambda_2) = (2,4c),
\end{equation*}
and is conjugate to $F(x) = x^2 + c$ (Milnor \cite{Milnor}).

\begin{exmp}
    Consider the multiplier polynomial $\Sigma_1(f)$ that has the pairs of multiplier matrix eigenvalues
    \begin{equation*}
        \{(0,0),(0,0), (3,21/2), (-1,3/2), (-1, 1/2), (3,-17/2), (0,-1)\}.
    \end{equation*}
    We assume these invariants are associated to a degree $2$ triangular polynomial endomorphism of the form
    \begin{align*}
        F:\P^2 &\to \P^2\\
        F(x,y,z) &= [F_1(x,z), F_2(x,y,z),z^2].
    \end{align*}
    Since the values of multipliers of (affine) fixed points of $F_1$ (as a polynomial endomorphism of $\P^1$) must occur in two distinct pairs of eigenvalues, the multipliers of the affine fixed points of $F_1$ must be in the set
    \begin{equation*}
        \{0,3,-1\}
    \end{equation*}
    Of the three possible pairs (plus the multiplier of $0$ at infinity), only one such pair gives a quadratic polynomial. It has Milnor parameters
    \begin{equation*}
        (\sigma_1,\sigma_2,\sigma_3) = (2, -3, 0) = (2,4c,0)
    \end{equation*}
    so we have $c=-3/4$ and $F_1(x,1) = x^2-3/4$ which (as an endomorphism of $\P^1$) has fixed points $\{-1/2, 3/2, \infty\}$.

    Associated to the fixed point $-1/2$, which has multiplier $-1$, are the eigenvalues $\{1/2,3/2,0\}$. These are the other value of the pairs of eigenvalues containing $-1$. Taking the possible pairs of these values, we get
    \begin{align*}
        (\lambda_1,\lambda_2) \to (\sigma_1, \sigma_2)\\
        (1/2, 3/2) \mapsto (2,3/4)\\
        (1/2, 0) \mapsto (1/2,0)\\
        (3/2, 0) \mapsto (3/2,0)
    \end{align*}
    The only one of these that is a quadratic polynomial is $(2,3/4)$ which corresponds to $F_2(1/2,y,1)=y^2 + 3/16$.

    Associated to the fixed point $3/2$, which has multiplier $3$, are the eigenvalues $\{21/2, -17/2\}$. The corresponding Milnor parameters are $(2, -357/4, 0)$. This gives the quadratic polynomial $F_2(-3/2,y,1) = y^2 - 357/16$.

    Associated to the fixed point at infinity, which has multiplier $0$, are the eigenvalues $\{0,3,-1\}$. The corresponding Milnor parameters are $(2,-3,0)$. This is the polynomial $F_2(1,y,0) = y^2-3/4$.

    We can conjugate so that the second coordinate of $F$ is a polynomial in normal form (which we will consider as lacking monomials $\{y^{d-1}x, y^{d-1}\}$ and with $y^d$ monic). Write the second coordinate as
    \begin{equation*}
        F_2(x,y,z) = y^2 + ax^2 + bxz + cz^2.
    \end{equation*}
    We need to solve
    \begin{align*}
        F_2(-1/2,y,1) &: 1/4a - 1/2b+c = 3/16\\
        F_2(3/2,y,1) &: 9/4a + 3/2b+c = -357/16\\
        F_2(1,y,0) &: a  = -3/4.
    \end{align*}
    The unique solution is $a=-3/4,b=-21/2, c=-39/8$, which gives the map
    \begin{equation*}
        F = [x^2 - 3/4z^2 : -3/4x^2 + y^2 - 21/2xz - 39/8z^2 : z^2].
    \end{equation*}
\begin{code}
\begin{python}
sage: R.<a,b,c>=QQ[]
sage: l1=a + 3/4
sage: l2=a*9/4 + b*3/2 + c + 357/16
sage: l3=a*1/4 - b*1/2 + c - 3/16
sage: I=R.ideal([l1,l2,l3])
sage: I.variety()
[{a: -3/4, b: -21/2, c: -39/8}]
\end{python}
\end{code}
\end{exmp}

\subsubsection{Monic Polynomials}
A \emph{regular polynomial endomorphism} is an endomorphism of $\P^N$ that leaves a hyperplane invariant.

Similar to Ingram \cite{Ingram5}, we restrict attention to morphisms $f : \P^2 \to \P^2$ of degree $d = 2$ with a totally invariant hyperplane $H \subset \P^2$ such that the restriction of $f$ to $H$ is the $d$th power map in some coordinates. This defines a subvariety $MP^2_2 \subset \calM_2^2$ of the space of coordinate-free endomorphisms of $\P^2$. Ingram calls such a map a \emph{monic polynomial} and we adopt his terminology. We next determine a model for conjugacy classes in $MP_2^2$.

Let $(x,y,z)$ be the coordinates of $\P^2$. We conjugate to move the invariant hyperplane to the hyperplane $z=0$ and, since we are assuming we have a monic polynomial, $\{(1,1,0), (0,1,0), (1,0,0)\}$ are fixed points. So the map is of the form
\begin{equation*}
    f(x,y,z) = (x^2+a_1xz+a_2yz + a_3z^2 : b_1xz+y^2+b_2yz+b_3z^2 : z^2].
\end{equation*}
We can conjugate by an element of the form
\begin{equation*}
    m = \begin{pmatrix}1&0&a\\
    0&1&b\\
    0&0&1\end{pmatrix}
\end{equation*}
to maintain the form of $f$.
\begin{code}
\begin{python}
sage: R.<a,a1,a2,a3,b,b1,b2,b3>=QQ[]
sage: P.<x,y,z>=ProjectiveSpace(R,2)
sage: #generic monic
sage: f=DynamicalSystem([x^2+a1*x*z+a2*y*z + a3*z^2, b1*x*z+y^2+b2*y*z+b3*z^2,z^2])
sage: m=matrix(FractionField(R), 3,3,[1,0,a, 0,1,b, 0,0,1])
sage: g=f.conjugate(m)
sage: g.normalize_coordinates();g
\end{python}
\end{code}
In particular,
\begin{align*}
    f^m = [x^2 &+ (2a + a_1)xz + a_2yz + (a^2 + aa_1 + a_2b - a + a_3)z^2 \\
        &: y^2 + b_1xz + (2b + b_2)yz + (b^2 + ab_1 + bb_2 - b + b_3)z^2 : z^2).
\end{align*}
So we have one more degree of freedom in this family. Choosing to also fix the point $(1,0,1)$ removes this freedom and forces $a_3 = -a_1$ and $b_3 =-b_1$. 
Thus, we can assume that our map is of the form
\begin{equation}\label{eq_normal_form}
    f(x,y,z) = [x^2 + a_1xz + a_2yz - a_1z^2 :  y^2+ b_1xz + b_2yz - b_1z^2 : z^2].
\end{equation}
With this form we can explicitly compute $\Sigma_1(f)$ and the relations among the multiplier invariants.

\begin{code}
\begin{python}
sage: R.<a1,a2,b1,b2>=QQ[]
sage: P.<x,y,z>=ProjectiveSpace(FractionField(R),2)
sage: f=DynamicalSystem([x^2 + a1*x*z + a2*y*z + (-a1)*z^2, b1*x*z + y^2 + b2*y*z + (-b1)*z^2,z^2])
sage: print f([1,1,0]), f([1,0,0]), f([0,1,0]), f([1,0,1])
sage: s=sigma(f,1, chow=False) #computation finishes in 3s

sage: # split them up
sage: w,t = s.parent().gens()
sage: N=2
sage: D=7
sage: L=[[0 for _ in range(N*D+1)] for _ in range(D+1)]
sage: for i in range(D-1,-1,-1):
sage:     for j in range(N*(D-i)+1):
sage:         if i == 0:
sage:             if j == 0:
sage:                 c = R(s.constant_coefficient())
sage:             else:
sage:                 c = R(s.subs({w:0}).coefficient({t:j}).subs({w:0,t:0}))
sage:         else:
sage:             c = R(s.coefficient({w:i,t:j}).subs({w:0,t:0}))
sage:         if c.degree()!=0:
sage:             if c.lc() <= 0:
sage:                 c = (-1)*c
sage:             L[D-i][N*(D-i)-j]=c

sage: #check dependence
sage: B=[]
sage: for i in range(len(L)):
sage:     ll = L[i]
sage:     for j in range(len(ll)):
sage:         l = ll[j]
sage:         I = R.ideal(B)
sage:         if l != 0:
sage:             if l not in I:
sage:                 print "good: ",(i,j)
sage:                 B.append(l)
sage:                 print l
sage:             else:
sage:                 print "redundant: ",(i,j)
sage: len(B)
\end{python}
\end{code}
To ease readability, we drop the superscript $(1)$ from the $\sigma_{i,j}^{(1)}$ in the following theorem.
\begin{thm}\label{thm_monic_poly}
    The image of $\tau_{2,1}^2$ restricted to monic polynomials of the form \eqref{eq_normal_form} is generated by
    \begin{align*}
        \sigma_{1,2} &= 8a_2b_1 + 4\\
        \sigma_{2,2} &= -2a_1^2 - 4a_1b_1 + 36a_2b_1 - 4a_2b_2 - 2b_2^2 - 4a_1 + 4a_2 -4b_1 + 4b_2 + 60\\
        \sigma_{2,3} &= 8a_1^2b_1 + 16a_1a_2b_1 + 16a_2b_1b_2 + 8a_2b_2^2 - 4a_1^2 + 8a_1b_1 + 40a_2b_1 - 24a_2b_2 - 4b_2^2 - 8a_1\\
        &+ 16a_2 + 8b_2 + 24\\
        \sigma_{2, 4} &= -4a_1^3b_1 + 18a_2^2b_1^2 - 8a_1a_2b_1b_2 - 2a_1^2b_2^2 - 4a_2b_2^3 - 4a_1^2b_1 + 24a_1a_2b_1 + 4a_1^2b_2 + 8a_2b_1b_2\\
         &- 4a_1b_2^2 + 20a_2b_2^2 - 4a_1^2 + 20a_2b_1 + 8a_1b_2 - 32a_2b_2 - 4b_2^2 - 8a_1 + 16a_2 + 8b_2\\
        \sigma_{3,3} &= 32a_1^2b_1 + 64a_1a_2b_1 - 8a_2^2b_1 - 8a_2b_1^2 + 64a_2b_1b_2 + 32a_2b_2^2 - 32a_1^2 + 128a_2b_1 - 128a_2b_2\\
         &- 32b_2^2 - 64a_1 + 96a_2 - 32b_1 + 64b_2 + 176.
    \end{align*}
    Further, the image of the (restricted) map
    \begin{align*}
        \tau_{2,1}^2: MP_2^2 &\to \A^5\\
        [f] &\mapsto (\sigma_{1,2},\sigma_{2,2},\sigma_{2,3}, \sigma_{2,4},\sigma_{3,3})
    \end{align*}
    is the hypersurface defined by the vanishing of
    \begin{align*}
        &36\sigma_{1,2}^5 - 18\sigma_{1,2}^4\sigma_{2,2} + 2\sigma_{1,2}^3\sigma_{2,2}^2 + 17712\sigma_{1,2}^4 - 8384\sigma_{1,2}^3\sigma_{2,2} + 1292\sigma_{1,2}^2\sigma_{2,2}^2 - 64\sigma_{1,2}\sigma_{2,2}^3 - 2456\sigma_{1,2}^3\sigma_{2,3}\\
        &+ 476\sigma_{1,2}^2\sigma_{2,2}\sigma_{2,3} + 73\sigma_{1,2}^2\sigma_{2,3}^2 +16\sigma_{1,2}\sigma_{2,2}\sigma_{2,3}^2
        +16\sigma_{1,2}^3\sigma_{2,4} + 792\sigma_{1,2}^3\sigma_{3,3} - 196\sigma_{1,2}^2\sigma_{2,2}\sigma_{3,3}\\
        &+ 8\sigma_{1,2}\sigma_{2,2}^2\sigma_{3,3} - 54\sigma_{1,2}^2\sigma_{2,3}\sigma_{3,3} - 4\sigma_{1,2}\sigma_{2,2}\sigma_{2,3}\sigma_{3,3} + 9\sigma_{1,2}^2\sigma_{3,3}^2 + 197280\sigma_{1,2}^3 - 105984\sigma_{1,2}^2\sigma_{2,2}\\
        &+ 22464\sigma_{1,2}\sigma_{2,2}^2
         -1792\sigma_{2,2}^3 + 48256\sigma_{1,2}^2\sigma_{2,3} - 12064\sigma_{1,2}\sigma_{2,2}\sigma_{2,3} - 11336\sigma_{1,2}\sigma_{2,3}^2 +1472\sigma_{2,2}\sigma_{2,3}^2\\
        &+ 512\sigma_{2,3}^3 - 51392\sigma_{1,2}^2\sigma_{2,4} + 20480\sigma_{1,2}\sigma_{2,2}\sigma_{2,4} - 2048\sigma_{2,2}^2\sigma_{2,4}
         +10240\sigma_{1,2}\sigma_{2,3}\sigma_{2,4}- 2048\sigma_{2,2}\sigma_{2,3}\sigma_{2,4}\\
        &- 512\sigma_{2,3}^2\sigma_{2,4} + 3008\sigma_{1,2}^2\sigma_{3,3} - 2400\sigma_{1,2}\sigma_{2,2}\sigma_{3,3} + 480\sigma_{2,2}^2\sigma_{3,3} + 2992\sigma_{1,2}\sigma_{2,3}\sigma_{3,3} - 240\sigma_{2,2}\sigma_{2,3}\sigma_{3,3}\\
        &- 256\sigma_{2,3}^2\sigma_{3,3}
        - 2560\sigma_{1,2}\sigma_{2,4}\sigma_{3,3} + 512\sigma_{2,2}\sigma_{2,4}\sigma_{3,3} + 256\sigma_{2,3}\sigma_{2,4}\sigma_{3,3} - 40\sigma_{1,2}\sigma_{3,3}^2 - 32\sigma_{2,2}\sigma_{3,3}^2\\
        &+ 32\sigma_{2,3}\sigma_{3,3}^2 - 32\sigma_{2,4}\sigma_{3,3}^2 + 2411904\sigma_{1,2}^2 - 1307136\sigma_{1,2}\sigma_{2,2}
        + 171968\sigma_{2,2}^2 + 268416\sigma_{1,2}\sigma_{2,3}\\
        &+ 16064\sigma_{2,2}\sigma_{2,3} - 38768\sigma_{2,3}^2 - 613632\sigma_{1,2}\sigma_{2,4} + 122880\sigma_{2,2}\sigma_{2,4} + 61440\sigma_{2,3}\sigma_{2,4}\\
        &+ 85376\sigma_{1,2}\sigma_{3,3} - 32320\sigma_{2,2}\sigma_{3,3} + 4768\sigma_{2,3}\sigma_{3,3} - 15360\sigma_{2,4}\sigma_{3,3}+ 1232\sigma_{3,3}^2 + 20517376\sigma_{1,2}\\
        &- 5436928\sigma_{2,2} - 459776\sigma_{2,3} - 1844224\sigma_{2,4}
        + 532480\sigma_{3,3} + 56702976.
    \end{align*}
\end{thm}
\begin{proof}
    The independence/dependence of the $\sigma_{i,j}$ is a simple ring calculation given their forms from $\Sigma_1(f)$, and it was performed in Sage.

    To get the hypersurface equation, we take the elimination ideal in the variables $(\ell_1,\ldots,\ell_5)$ of $(\sigma_{1,2} - \ell_1, \ldots,\sigma_{3,3}-\ell_5)$. The statement then follows from the Elimination and Extension Theorems \cite[Chapter 3]{CLO}.
\end{proof}

\begin{code}
\begin{python}
sage: RR.<l1,l2,l3,l4,l5>=PolynomialRing(R, order='lex')
sage: phi=RR.flattening_morphism()
sage: pJ=phi.codomain().ideal([phi(B[0]-l1), phi(B[1]-l2), phi(B[2]-l3), phi(B[3]-l4), phi(B[4]-l5)])
sage: G=pJ.elimination_ideal(phi.codomain().gens()[:4])
sage: G.gens()
\end{python}
\end{code}


\begin{cor}\label{cor_monic_finite}
    The fixed point multiplier map $\tau_{2,1}^2$ is (generically) finite-to-one when restricted to monic polynomials of the form \eqref{eq_normal_form}.
\end{cor}
\begin{proof}
    Consider the map
    \begin{align*}
        \phi: \A^4 \to X \subset \A^5\\
        (a_1,a_2,b_1,b_2) &\mapsto (\sigma_{1,2},\sigma_{2,2},\sigma_{2,3},\sigma_{2,4},\sigma_{3,3}),
    \end{align*}
    where $X$ is the hypersurface defined in Theorem \ref{thm_monic_poly}. Then $\phi$ is a dominant morphism of affine varieties. Hence, there is an open set $U \subset X$ such that
    \begin{equation*}
        \dim(\phi^{-1}(y)) = \dim(\A^4) - \dim(X) = 0, \quad \forall y \in U.
    \end{equation*}
    In particular, $\tau_{2,1}^2$ is finite-to-one.
\end{proof}
Note that Corollary \ref{cor_monic_finite} does not give the generic degree of $\tau_{2,1}^2$ and the computations to compute the degree did not finish on the machine being used. However, in some partial calculations (fixing a subset of the $\sigma_{i,j}^{(1)}$), it appears the generic degree should be $8$. Although, there were a number of closed subsets where the degree is $12$. Note also that when we used only the invariants defined in Doyle-Silverman \cite{JDoyle2} ($\{\sigma_{D_1,j}^{(1)} : 1 \leq j \leq ND_1\}$), then the computations similar to Theorem \ref{thm_monic_poly} did not finish in a reasonable amount of time even though one would still expect to find five generators and a hypersurface requirement.


\begin{thebibliography}{10}

\bibitem{Abate}
M.~Abate.
\newblock Index theorems for meromorphic self-maps of the projective space.
\newblock In {\em Frontiers in complex dynamics}, volume~51 of {\em Princeton
  Math. Ser.}, pages 451--462. Princeton University Press, 2014.

\bibitem{CLO}
David Cox, John Little, and Donal O'Shea.
\newblock {\em Ideals, Varieties, and Algorithms}.
\newblock Undergraduate Texts in Mathematics. Springer-Verlag, 1991.

\bibitem{DeMarco}
Laura DeMarco.
\newblock The moduli space of quadratic rational maps.
\newblock {\em J. Amer. Math. Soc.}, 20:321--355, 2007.

\bibitem{Dias}
Anna Paula~S. Dias and Ian Stewart.
\newblock Invariant theory for wreath product groups.
\newblock {\em Journal of Pure and Applied Algebra}, 150:61--84, 2000.

\bibitem{JDoyle2}
John~R. Doyle and Joseph~H. Silverman.
\newblock Moduli spaces for dynamical systems with portraits.
\newblock {\em arxiv:1812.0993}, 2018.

\bibitem{Fujimura2}
Masayo Fujimura.
\newblock The moduli space of rational maps and surjectivity of multiplier
  representation.
\newblock {\em Comput. Methods Funct. Theory}, 7(2):345--360, 2007.

\bibitem{Fujimura}
Masayo Fujimura and Kiyoko Nishizawa.
\newblock {\em The real multiplier coordinate space of the quartic
  polynomials}, pages 61--69.
\newblock Yokohama Publ., 2007.

\bibitem{Hutz18}
Thomas Gauthier, Benjamim Hutz, and Scott Kaschner.
\newblock Symmetrization of rational maps: arithmetic properties and families
  of {L}att{\`e}s maps of {Pk}.
\newblock {\em arXiv:1603.04887}, submitted 2016.

\bibitem{ghioca4}
Dragos Ghioca and Khoa~D. Nguyen.
\newblock Dynamics of split polynomial maps: uniform bounds for periods and
  applications.
\newblock {\em Int. Math. Res. Not.}, 2017:213--231, 2017.

\bibitem{Guillot}
Adolfo Guillot.
\newblock Un th{\'e}or{\`e}me de point fixe pour les enomorphismes de l'espace
  projectif avec des applications aux feuilletages alg{\'e}briques.
\newblock {\em Bull Braz Math Soc, New Series}, 35(3):345--362, 2004.

\bibitem{Guillot3}
Adolfo Guillot.
\newblock Semicompleteness of homogeneous quadratic vector fields.
\newblock {\em Ann. Inst. Fourier (Grenoble)}, 56(5):1583--1615, 2006.

\bibitem{Guillot4}
Adolfo Guillot.
\newblock Quadratic differential equations in three variables without
  multivalued solutions: Part i.
\newblock {\em SIGMA Symmetry Integrability Geom. Mathods Appl. 14}, 122:46
  pages, 2018.

\bibitem{Guillot2}
Adolfo Guillot and Valente Ram{\'i}rez.
\newblock On the multipliers at fixed points of self-maps of the projective
  plane.
\newblock {\em arxiv:1902.04433}, 2019.

\bibitem{Hutz10}
Benjamim Hutz and Michael Tepper.
\newblock Multiplier spectra and the moduli space of degree 3 morphisms on
  {P1}.
\newblock {\em JP Journal of Algebra, Number Theory and Applications},
  29(2):189--206, 2013.

\bibitem{Hutz1}
Benjamin Hutz.
\newblock Dynatomic cycles for morphisms of projective varieties.
\newblock {\em New York J. Math}, 16:125--159, 2010.

\bibitem{Ingram5}
Patrick Ingram.
\newblock Rigidity and height bounds for certain post-critically finite
  endomorphisms of projective space.
\newblock {\em Canad. J. Math.}, 68:625--654, 2016.

\bibitem{Levy}
Alon Levy.
\newblock The space of morphisms on projective space.
\newblock {\em Acta Arithmetica}, 146:12--31, 2011.

\bibitem{Levy4}
Alon Levy.
\newblock The {M}c{M}ullen map in positive characteristic.
\newblock {\em arxiv:1304.2804}, 2013.

\bibitem{Manes3}
Michelle Manes.
\newblock Moduli spaces for families of rational maps on {P1}.
\newblock {\em J. Number Theory}, 127(7):1623--1663, 2009.

\bibitem{McMullen2}
Curtis McMullen.
\newblock Families of rational maps and iterative root-finding algorithms.
\newblock {\em Ann. of Math.}, 125(3):467--493, 1987.

\bibitem{Milnor}
J.~Milnor.
\newblock Geometry and dynamics of quadratic rational maps.
\newblock {\em Experiment. Math.}, 2(1):37--83, 1993.

\bibitem{Milnor3}
J.~Milnor.
\newblock {\em Complex Dynamics in One Variable}, volume 160 of {\em Annals of
  Mathematical Studies}.
\newblock Princeton University Press, 2006.

\bibitem{Minimair}
Manfred Minimair.
\newblock Dense resultant of composed polynomials: Mixed-mixed case.
\newblock {\em Journal of Symbolic Computation}, 36(6):825--834, December 2003.

\bibitem{petsche}
Clayton Petsche, Lucien Szpiro, and Michael Tepper.
\newblock Isotriviality is equivalent to potential good reduction for
  endomorphisms of {$\mathbb{P}^N$} over function fields.
\newblock {\em Journal of Algebra}, 322(9):3345--3365, 2009.

\bibitem{Rong}
Feng Rong.
\newblock Latt{\`e}s maps on {$\mathbb{P}^2$}.
\newblock {\em J. Math Pures Appl.}, 93:636--650, 2010.

\bibitem{Silverman9}
Joseph~H. Silverman.
\newblock The space of rational maps on {P1}.
\newblock {\em Duke Math. J.}, 94:41--118, 1998.

\bibitem{Silverman10}
Joseph~H. Silverman.
\newblock {\em The Arithmetic of Dynamical Systems}, volume 241 of {\em
  Graduate Texts in Mathematics}.
\newblock Springer-Verlag, 2007.

\bibitem{Silverman20}
Joseph~H. Silverman.
\newblock {\em Moduli Spaces and Arithmetic Dynamics}.
\newblock {CRM} Monograph Series. American Mathematical Society, Providence,
  RI, 2012.

\bibitem{Sugiyama2}
Toshi Sugiyama.
\newblock The moduli space of polynomial maps and their fixed-point
  multipliers.
\newblock {\em Adv. Math.}, 322:132--185, 2017.

\bibitem{Ueda2}
Tetsuo Ueda.
\newblock Complex dynamics on projective spaces -- index formula for fixed
  points.
\newblock {\em Proceedings of the international conference on dynamical systems
  and chaos}, 1:252--259, 1995.

\end{thebibliography}
\end{document}